\documentclass{amsart} 

\usepackage{amsmath,amssymb,latexsym, amscd}
\usepackage{exscale}
\usepackage{cite}
\usepackage{epsfig}
\usepackage{graphics}

\renewcommand{\geq}{\geqslant}
\renewcommand{\leq}{\leqslant}

\renewcommand{\L}{\mathcal{L}} 
\renewcommand{\l}{\langle}
\renewcommand{\r}{\rangle}

\DeclareMathOperator{\spn}{span}
\DeclareMathOperator{\ran}{range}

\DeclareMathOperator{\tr}{tr}

\newtheorem{theorem}{Theorem}[section]
\newtheorem{lemma}[theorem]{Lemma}
\newtheorem{proposition}[theorem]{Proposition}
\newtheorem{remark}[theorem]{Remark}

\newtheorem{corollary}[theorem]{Corollary}
\newtheorem{assumption}[theorem]{Assumption}

\newtheorem*{main-theorem}{Main Theorem}
\newtheorem*{remark*}{Remark}
\numberwithin{equation}{section}

\title[Modulational instability]{Modulational instability \\ and variational structure}

\author[Bronski]{Jared~C.~Bronski}
	\address{Department of Mathematics, University of Illinois at Urbana-Champaign,
	1409 W Green Street, Urbana, IL 61801}
	\email{jared@\allowbreak math.\allowbreak uiuc.\allowbreak edu}

\author[Hur]{Vera~Mikyoung~Hur}
	\email{verahur@\allowbreak math.\allowbreak uiuc.\allowbreak edu}

\date{\today}

\begin{document}

\maketitle

\begin{abstract}
We study the modulational instability of periodic traveling waves 
for a class of Hamiltonian systems in one spatial dimension. 
We examine how the Jordan block structure of the associated linearized operator 
bifurcates for small values of the Floquet exponent  
to derive a criterion governing instability to long wavelengths perturbations 
in terms of the kinetic and potential energies, the momentum, the mass 
of the underlying wave, and their derivatives. 
The dispersion operator of the equation is allowed to be nonlocal, 
for which Evans function techniques may not be applicable. 
We illustrate the results by discussing analytically and numerically 
equations of Korteweg-de Vries type.
\end{abstract}

\section{Introduction}\label{S:intro}
We study the stability and instability of periodic traveling waves 
for a class of Hamiltonian systems in one spatial dimension, 
in particular, equations of Korteweg-de Vries (KdV) type
\begin{equation}\label{E:KdV1} 
u_t-\mathcal{M}u_x+f(u)_x=0
\end{equation}
in the theory of wave motion. Here $t \in \mathbb{R}$ is typically proportional to elapsed time 
and $x \in \mathbb{R}$ is usually related to the spatial variable in the primary direction 
of wave propagation; $u=u(x,t)$ is real valued, frequently representing the wave profile or a velocity. 
Throughout we express partial differentiation either by a subscript or using the symbol $\partial$. 
Moreover $\mathcal{M}$ is a Fourier multiplier, defined as 
$\widehat{\mathcal{M}u}(\xi)=m(\xi)\hat{u}(\xi)$ 
and characterizing dispersion in the linear limit, while $f$ describes the nonlinearity. 
In many examples of interest, $f$ obeys a power law. 

Perhaps the best known among equations of the form \eqref{E:KdV1} is the KdV equation 
\[
u_t+u_{xxx}+(u^2)_x=0
\] 
itself, which was put forward in \cite{Boussinesq} and \cite{KdV} 
to model the unidirectional propagation of surface water waves 
with small amplitudes and long wavelengths in a channel; 
it has since found relevances in other situations such as Fermi-Pasta-Ulam lattices 
(see \cite{FPU}, for instance). Observe, however, that 
\eqref{E:KdV1} is {\em nonlocal} unless the dispersion symbol $m$ is a polynomial of $i\xi$;
examples include the Benjamin-Ono equation (see \cite{Benjamin, Ono}, for instance) 
and the intermediate long wave equation (see \cite{Joseph}, for instance), 
for which $m(\xi)=|\xi|$ and $\xi\coth\xi-1$, respectively, while $f(u)=u^2$. 
Another example, proposed by Whitham \cite{Whitham} to argue for breaking of water waves, 
corresponds to $m(\xi)=\sqrt{\tanh\xi\,/\xi}$ and $f(u)=u^2$. 
Incidentally the quadratic power-law nonlinearity is characteristic of many wave phenomena. 

A traveling wave solution of \eqref{E:KdV1} takes the form $u(x,t)=u(x+ct)$, 
where $c\in \mathbb{R}$ and $u$ satisfies by quadrature that
\begin{equation}\label{E:tKdV1}
\mathcal{M}u-f(u)-cu-a=0
\end{equation}
for some $a\in \mathbb{R}$. 
In other words, it steadily propagates at a constant speed without changing the configuration. 
For a broad range of dispersion symbols and nonlinearities, 
a plethora of periodic traveling waves of \eqref{E:KdV1} may be attained from variational arguments, 
e.g., the mountain pass theorem applied to a suitable functional 
whose critical point satisfies \eqref{E:tKdV1}. The associated spectral problem 
\[
\mu v=\mathcal{M}v_x-(f'(u)v)_x-cv_x
\]
is the subject of investigation here. 

\medskip

As Alan Newell explained in \cite{Newell}, ``if dispersion and nonlinearity act against each other, 
monochromatic wave trains do not wish to remain monochromatic. 
The sidebands of the carrier wave can draw on its energy via a resonance mechanism, 
with the result that the envelop becomes modulated." 
Benjamin and Feir \cite{BF} and Whitham \cite{Whitham1967} formally argued that 
Stokes' periodic waves at the surface of deep water would be unstable, 
leading to sidebands growth, namely the {\em modulational} or Benjamin-Feir instability. 
Corroborating results arrived nearly simultaneously, albeit independently, 
by Lighthill \cite{Lighthill}, Ostrovsky \cite{Ostrovsky}, Benney and Newell \cite{BN}, 
Zakharov \cite{Zakharov-NLS, Zakharov-WW}, among others;
see also \cite{Whitham} and references therein. 
Modulational instability occurs in numerous physical systems, other than water waves, 
such as optics (see \cite{Ostrovsky, Zakharov-NLS, AL, TTH, HK}, for instance) 
and plasmas (see \cite{Hasegawa, McB}, for instance). Furthermore it results in 
various nonlinear processes such as envelop solitons, dispersive shocks and rogue waves.

\medskip

Recently a great deal of work has aimed at translating 
formal modulation theories in \cite{Whitham}, for instance, into rigorous mathematical results. 
It would be impossible to do justice to all advances in the direction, but we may single out a few 
--- \cite{OZ2003a, OZ2003b, Serre2005} for conservation laws with viscosity, 
\cite{DS2009} for nonlinear Schr\"odinger equations, 
and \cite{BJ2010, JZ2010} for equations of KdV type; 
see also \cite{BM1995} for the Benjamin-Feir instability of Stokes waves in water of finite depth.

In particular in \cite{BJ2010}, a rigorous calculation of long wavelengths perturbations 
was made for (local) KdV equations with general nonlinearities 
--- henceforth called generalized KdV equations --- 
via Evans function techniques as well as a Bloch wave decomposition. 
Under certain nondegeneracy conditions, in fact, 
the spectrum of the associated linearized operator in the vicinity of the origin was shown to take 
a normal form --- either the spectrum consists of the imaginary axis with multiplicity three
or it contains three lines through the origin, one in the imaginary axis and two in other directions;  
the latter implies instability. Furthermore the normal form was determined by an index, 
which was effectively calculated in terms of conserved quantities of the PDE 
and their derivatives with respect to constants of integration arising in the traveling wave ODE.
 
 \medskip
 
Here we take matters further and derive a criterion governing spectral instability near the origin 
of periodic traveling waves for a general class of Hamiltonian systems in one\footnote{
The requirement that the equation is in one spatial dimension merely enters in the discussion 
of a Pohozaev type identity, which is to avoid inversion of a linearized operator; 
see Lemma~\ref{L:pohozaev1}.} 
spatial dimension. We shall make a few assumptions --- 
mainly the existence of a conserved momentum and a Casimir invariant, interpreted as the mass --- 
but do not otherwise restrict the form of the equation, 
considerably broadening the scope of applications. 
Of particular interest are nonlocal equations, for which Evans function techniques\footnote{
Note however that the approach in \cite{GLM2007,GLZ2008}, for instance, 
realizing the Evans function as a regularized Fredholm determinant may be 
more generally applicable than the standard ODE based formulation.} and other ODE methods 
(which are instrumental in \cite{BJ2010}, for instance, in the derivation of index formulae) 
may not be applicable. Instead we perform a spectral perturbation 
of the associated linearized operator with respect to the Floquet exponent 
and replace ODE based arguments by functional analytic ones. 
Variational properties of the equation will help us to calculate index formulae
without recourse to the small amplitude wave limit. 
Incidentally Lin \cite{Lin-JFA} devised a continuation argument 
and generalized the stability theory of solitary waves in \cite{GSS, BSS, PW1992}, among others, 
to a class of nonlinear nonlocal equations.

\medskip

Our results are most explicit in the case of KdV type equations with fractional dispersion, 
which we work out in Section~\ref{S:KdV} and Section~\ref{S:gKdV}. 
In particular we calculate the modulational instability index in terms of 
the kinetic and potential energies, the momentum and the mass of the underlying wave, 
together with their derivatives with respect to Lagrange multipliers 
arising in the traveling wave equation as well as the wave number. 
In the case of the quadratic power-law nonlinearity we further express the index 
in terms of the potential energy, the momentum and the mass as functions of 
Lagrange multipliers associated with conservations of the momentum and the mass. 
We conduct numerical experiments in Section~\ref{S:numerical}.

\section{Abstract framework}\label{S:theory}
We shall derive a sufficient condition of spectral instability to long wavelengths perturbations
 of periodic traveling waves, for a class of Hamiltonian systems in one spatial dimension, 
under a few assumptions; they will be stated as they are needed. 

\subsection{Preliminaries}\label{SS:preliminaries}
Consider a Hamiltonian system of the form  
\begin{equation}\label{E:equation}
u_t=J\delta H(u),
\end{equation}
where $J$ is a linear skew-symmetric operator, independent of $u$, $H$ is a Hamiltonian 
and $\delta$ denotes variational differentiation. 

Throughout we work in the $L^2$-Sobolev spaces setting. 
We employ the notation $\langle \,, \rangle$ for the $L^2$-inner product. 

\begin{assumption}[Conservation laws]\label{A:HPM}\rm
Assume that \eqref{E:equation} possesses in addition to $H$ two conserved quantities, 
denoted $P$ and $M$. Assume that $H$, $P$, $M$ are smooth in an appropriate function space 
and invariant under spatial translations. Moreover assume that 
\begin{itemize}
\item[(P)] $J \delta P(u) = u_x$,
\item[(M)] $\ker(J)=\spn\{\delta M(u)\}$.
\end{itemize} 
\end{assumption}

We refer to $P$ and $M$ as the momentum and the mass, respectively. 
Assumption {(P)} states that $P$ generates spatial translations while 
assumption {(M)} implies that $M$ is a Casimir invariant of the flow induced by \eqref{E:equation}. 
Thanks to Noether's theorem, conservation of the momentum is expected 
whenever \eqref{E:equation} is invariant under spatial translations. 

\medskip

Clearly \eqref{E:KdV1} satisfies Assumption~\ref{A:HPM}, 
for which  $J=\partial_x$, 
\begin{align*}
H =& \int \Big(\frac12 u {\mathcal M} u-F(u)\Big)~dx, \qquad \text{where }F'=f, \\  
P =& \int \frac12 u^2~dx, \\  
M =& \int u~dx.
\end{align*}
More generally,
\[
Lu_t-{\mathcal M}u_x+f(u)_x=0, \qquad \text{where $L$ is a Fourier multiplier,}
\] 
satisfies Assumption~\ref{A:HPM}, for which $J=L^{-1}\partial_x$,
\begin{gather*}
H = \int \Big(\frac12 u {\mathcal M} u-F(u)\Big)~dx, \\
P = \int \frac12uLu~dx\quad\text{and}\quad M = \int u~dx.
\end{gather*}
Examples include the Benjamin-Bona-Mahony equation (see \cite{BBM}, for instance),
for which $L=1-\partial_x^2$, $\mathcal{M}=-\partial_x^2$ and $f(u)=u^2$.

\medskip

A traveling wave solution of \eqref{E:equation} takes the form $u(x,t)=u(x+ct+x_0)$, where 
$c\in\mathbb{R}$ represents the wave speed, $x_0\in\mathbb{R}$ is the spatial translate 
and $u$ satisfies that 
\begin{equation}\label{E:traveling} 
cu_x=J\delta H(u).
\end{equation}
Equivalently, $u$ arises as a critical point of 
\begin{equation}\label{D:E}
E(u; c,a ):=H(u)-cP(u)-aM(u)
\end{equation} 
for some $a \in \mathbb{R}$. That is to say, it satisfies that
\begin{equation}\label{E:traveling'} 
\delta E(u;c,a)=0.
\end{equation}
Indeed, since 
\[
\langle\delta M(u), u_x\rangle = \int_\mathbb{R} M(u)_x~dx = 0 
\] 
it follows from (M) that $u_x$ is orthogonal to $\ker(J)$. 
Applying $J^{-1}$ to \eqref{E:traveling} we then find that (P) implies \eqref{E:traveling'}.

\begin{assumption}[Periodic traveling waves]\label{A:periodic}\rm
Assume that \eqref{E:equation} admits a smooth, four-parameter family of periodic traveling waves, 
denoted $u(\cdot+x_0;c,a,T)$, which satisfies \eqref{E:traveling}, or equivalently \eqref{E:traveling'}, 
and is $T$-periodic for some $T>0$, the period. Assume that $u$ is even.
\end{assumption}

The existence of periodic traveling waves of \eqref{E:equation} usually follows 
from variational arguments. To illustrate, we shall discuss in Proposition~\ref{P:existence} 
minimization problems for a family of KdV equations with fractional dispersion. 
The symmetry assumption is to break that \eqref{E:equation} is invariant under spatial translations.

\medskip

Differentiating \eqref{E:traveling'} with respect to $x_0$ and $c, a$, respectively, 
we use {(P)} and {(M)} to obtain that
\begin{subequations}
\begin{alignat}{2}
&\delta^2 E u_{x_0} =0,  \qquad \qquad &&J \delta^2 E u_{x_0} = 0, \label{E:H1}\\
&\delta^2 E u_c = \delta P, \qquad \qquad &&J \delta^2 E u_c = u_x = u_{x_0}, \label{E:H2} \\
&\delta^2 E u_a = \delta M,\qquad \qquad &&J \delta^2 E u_a = 0. \label{E:H3}
\end{alignat}
\end{subequations}
Furthermore
\begin{equation}\label{E:PaMc}
M_c = \langle\delta M, u_c\rangle = \langle\delta^2 E u_a, u_c\rangle 
= \langle u_a, \delta^2 E u_c\rangle = \langle u_a, \delta P\rangle = P_a.
\end{equation} 

\begin{remark}\rm
Perhaps \eqref{E:H1} through \eqref{E:H3} are familiar to readers from thermodynamics, 
where the free energy --- $E$ in the present setting --- serves as a generating function of 
various quantities of interest. They are in fact found as derivatives of the free energy
with respect to Lagrange multipliers for a suitable variational problem. 
The equality of mixed partial derivatives then leads to relations among their derivatives, 
known as Kirchhoff's equations.      
\end{remark}

\begin{remark}\label{R:uT}\rm
The period $T$ enters calculations in a slightly different manner from 
other, periodic traveling wave parameters $x_0$, $c$, $a$. Although $\delta^2 E u_T =0$, 
formally, i.e., with the set of smooth functions as the domain of $\delta^2 E$, nevertheless, 
$u_T$ is not in general $T$-periodic. Rather $u_T$ exhibits a secular growth linear in $x$. 
Later we shall take a linear combination of $u_T$ and $xu_x$ to develop a Pohozaev type identity. 
\end{remark}

Here and in the sequel, we may regard $H$, $P$, $M$, 
evaluated at a periodic traveling wave $u(\cdot+x_0; c, a, T)$ of \eqref{E:equation} 
and restricted to one period, as functions of $c$ and $a$, the Lagrange multipliers 
associated with conservations of the momentum and the mass, respectively, 
as well as $T$, the period. Therefore we are permitted to differentiate $H$, $P$, $M$ 
with respect to the periodic traveling wave parameters $c$,~$a$,~$T$. 
We shall ultimately derive a modulational instability index in terms of $H$, $P$, $M$ 
and their derivatives with respect to $c$, $a$, $T$. Corresponding to translational invariance, 
$x_0$ plays no significant role in the present development. Hence we may mod it out.

\medskip

Linearizing \eqref{E:equation} about a periodic traveling wave $u=u(\cdot\,; c, a, T)$ 
in the frame of reference moving at the speed $c$, we arrive at that
\begin{equation}\label{E:linear} 
v_t=J\delta^2E(u; c, a)v=:\mathcal{L}(u; c, a)v.
\end{equation}
Seeking a solution of the form $v(x,t)=e^{\mu t}v(x)$, moreover, we arrive at the spectral problem
\begin{equation}
\label{E:eigen} \mu v=\mathcal{L}(u;c,a)v.
\end{equation}
We then say that $u$ is (spectrally) {\em unstable} if 
the $L^2(\mathbb{R})$-spectrum of $\mathcal{L}$ intersects 
the open, right half plane of $\mathbb{C}$. 
Note that $v$ needs not have the same period as $u$, namely a sideband perturbation.

\medskip

In the case of the (local) generalized KdV equation
\begin{equation}\label{E:gKdV}
u_t+u_{xxx}+f(u)_x=0,\qquad \text{where $f$ is a nonlinearity},
\end{equation}
the $L^2(\mathbb{R})$-spectrum of the associated linearized operator was related 
in \cite{BJ2010}, for instance, to eigenvalues of the monodromy map (or the periodic Evans function), 
and its characteristic polynomial led to two stability indices. 
The first index counts modulo two the total number of positive eigenvalues in the periodic functions setting and 
it extends that in \cite{GSS,BSS,PW1992}, among others, governing stability of solitary waves. 
The second index, on the other hand, furnishes a sufficient condition of instability 
to long wavelengths perturbations and it justifies the formal modulation theory 
in \cite{Whitham}, for instance. The present purpose is to extend the second index 
to a general class of Hamiltonian systems allowing nonlocal dispersion, 
for which Evans function techniques and other ODE methods may not be applicable. 
Instead we rely upon a Bloch wave decomposition of the related spectral problem. 
Lin \cite{Lin-JFA} devised a continuation argument and extended the first index 
to solitary waves for a class of nonlinear nonlocal equations;
see \cite{HJ1}, for instance, for an adaptation in the periodic wave setting.  

\medskip

It is standard from Floquet theory (see \cite{chicone}, for instance) that 
any eigenfunction of \eqref{E:eigen} takes the form 
\[ 
v(x)= e^{i \tau x } \phi(x), \qquad \text{where $\phi$ is $T$-periodic and $\tau \in (-\pi/T,\pi/T]$}
\]
denotes the Floquet exponent.
Accordingly \eqref{E:eigen} leads to the one-parameter family of spectral problems 
\begin{equation}\label{E:eigen-tau}
\mu \phi=e^{-i \tau x} \L(u;c,a) e^{i \tau x} \phi=: \mathcal{L}_\tau(u;c,a) \phi,
\end{equation}
suggesting us to study $L^2_{per}([0,T])$-spectra of $\L_\tau$. 
Notice that the spectrum of $\L_\tau$ consists merely of discrete eigenvalues. Furthermore
\[
{\rm spec}_{L^2(\mathbb{R})}(\L)=\bigcup_{\tau}{\rm spec}_{L^2_{per}([0,T])}(\L_\tau).
\]

One does not expect to be able to compute the spectrum of $\mathcal{L}_\tau$ 
for an arbitrary~$\tau$, however, except in few special cases, e.g., completely integrable systems 
(see \cite{BD2009}, for instance). Instead we are going to restrict the attention to 
the Floquet exponent $\tau$ and the eigenvalue $\mu$ both small. Physically this amounts to 
long wavelengths perturbations or slow modulations of the underlying wave. 
We shall first study the spectrum of the unmodulated operator $\L_0=\L$ at the origin; 
zero is an eigenvalue of $\L_0$, thanks to the latter equations in \eqref{E:H1} and \eqref{E:H3}. 
We shall then examine how the spectrum near the origin of the modulated operator $\L_\tau$ 
bifurcates from that of $\L_0$ for $|\tau|$ small. 

\medskip

We promptly discuss ``nondegeneracy" assumptions for a periodic traveling wave 
of \eqref{E:equation}, under which the generalized $L^2_{per}([0,T])$-null space of $\L_0=\L$ 
supports a Jordan block structure. 

\begin{assumption}[Nondegeneracies]\label{A:kerL0}\rm
Assume that
\begin{itemize}
\item[(N1)] $u(\cdot\,;c,a,T)$ is not constant; 
\item[(N2)] $\ker(\delta^2E(u;c,a))=\spn\{u_x\}$;
\item[(N3)] $G:={M}_c {P}_a -{M}_a {P}_c \neq 0$.
\end{itemize}
\end{assumption}
In what follows, we employ the notation 
\begin{equation}\label{D:poisson-bracket} 
\{f, g\}_{x,y}=f_xg_y-f_yg_x
\end{equation}
and write $G=\{M,P\}_{c,a}$.

\medskip

Assumption {(N1)} states that the underlying, periodic traveling wave is nondegenerate. 
It is not a serious assumption since if the profile of the underlying wave is constant 
then its stability proof is easy. In the case of the quadratic power-law nonlinearity, 
for which the related, time dependent equation obeys Galilean invariance, 
this amounts to understanding the stability of the zero state.

\medskip

Assumption {(N2)} states the nondegeneracy of the linearized operator 
associated with the traveling wave equation; that is to say, 
the kernel is spanned merely by spatial translations. 
It proves a spectral condition, which plays a central role in the stability of traveling waves 
(see \cite{Weinstein1987, Lin-JFA}, among others) and the blowup
(see \cite{KMR}, for instance) for the related, time dependent equation, 
and it necessitates a proof. Actually one may concoct a polynomial nonlinearity, say, $f$, 
for which the kernel of $-\partial_x^2-f'(u)$ at the underlying, periodic traveling wave 
is two dimensional at isolated points. 

In the case of generalized KdV equations (see \eqref{E:gKdV}), 
the nondegeneracy of the linearization at a periodic traveling wave was characterized  
in \cite{BJ2010}, for instance, through the wave amplitude as a function of the period.
Furthermore it was verified in \cite{Kwong}, among others, at ground states in all dimensions. 
By a ground state, incidentally, we mean a traveling wave solution 
that is positive and radial and which vanishes asymptotically.
The proofs rely upon shooting arguments  and the Sturm-Liouville theory for ODEs, 
which may not be applicable to nonlocal equations. Nevertheless, 
Frank and Lenzmann \cite{FL} recently obtained {(N2)} of Assumption~\ref{A:kerL0} 
at ground states for a family of nonlinear nonlocal equations. 
We shall adapt it in Proposition~\ref{P:kernel} in the periodic wave setting. 

\medskip

Assumption {(N3)} implies that the mapping $(c,a)\mapsto (P,M)$ is of $C^1$ and locally invertible,
namely the nondegeneracy of the constraint set for the periodic, traveling wave equation. 
We shall achieve it in Lemma~\ref{L:negativity} in the case of KdV equations with fractional dispersion 
near the solitary wave limit. Note in passing that $G$ may vanish along a curve of co-dimension one. 
We shall in fact indicate that a change in the sign of $G$ signals an eigenvalue 
of $\delta^2 E$ crossing from the left half plane of $\mathbb{C}$ to the right through the origin.  

\begin{lemma}[Jordan block structure]\label{L:L0}
Under Assumption~\ref{A:HPM}, Assumption~\ref{A:periodic} and Assumption~\ref{A:kerL0}, 
the generalized $L^2_{per}([0,T])$-null space of $\L=\L_0(u;c,a)$, defined in \eqref{E:linear}, 
at a periodic traveling wave $u=u(\cdot\,;c,a,T)$ of \eqref{E:equation}, 
possesses the Jordan block structure: 
\begin{itemize}
\item[(J1)] $\dim(\ker(\L_0)) = 2$,
\item[(J2)] $\dim(\ker(\L_0^2)/\ker(\L_0)) = 1$, 
\item[(J3)] $\dim(\ker(\L_0^{n+1})/\ker({\mathcal L_0^{n}} )) = 0$ for $n\geq2$ an integer.  
\end{itemize}
Furthermore  
\begin{subequations}
\begin{alignat}{2}
v_1 &:= u_a, \qquad \qquad w_1 &&:= M_c \delta P - P_c \delta M,\label{D:vw1}\\
v_2 &:= u_x, \qquad \qquad w_2 &&:= J^{-1} (M_a u_c - M_cu_a), \label{D:vw2}\\
v_3 &:= u_c, \qquad \qquad w_3 &&:=  P_a \delta M - M_a \delta P \label{D:vw3}
\end{alignat}\end{subequations}
form a basis and a dual basis, respectively:
\begin{subequations}
\begin{alignat}{2}
&\L_0 v_1 = 0, \qquad \qquad &&\L_0^\dagger w_1=0, \label{E:vw1}\\
&\L_0 v_2 = 0, \qquad \qquad &&\L_0^\dagger w_2 = w_3,  \label{E:vw2}\\
&\L_0 v_3 = v_2, \qquad \qquad &&\L_0^\dagger w_3=0. \label{E:vw3}
\end{alignat}\end{subequations}
Here and elsewhere, the dagger means adjoint. Moreover
\begin{equation}\label{E:G}
\langle w_j,v_k\rangle =(M_cP_a-M_aP_c)\delta_{jk}=G\delta_{jk}, 
\end{equation}
where $\delta_{jk}=\begin{cases} 1\quad \text{if } j=k,\\ 0 \quad \text{if }j\neq k.\end{cases}$ 
\end{lemma}

\begin{proof}
Note from the latter equations in \eqref{E:H1} and \eqref{E:H3} that $v_1, v_2 \in \ker(\L_0)$. 
Note from Assumption~\ref{A:periodic} that $v_1, v_3, w_1, w_3$ are even functions 
and $v_2, w_2$ are odd. Since $v_2\not\equiv 0$ by (N1) of Assumption~\ref{A:kerL0}, 
we infer from the former equations in \eqref{E:H1} and \eqref{E:H3} that 
$v_1$ and $v_2$ are linearly independent. 
Since $\ker(\L_0)$ is at most two dimensional by (N2) of Assumption~\ref{A:kerL0},  
furthermore, $\ker(\L_0) = \spn\{v_1,v_2\}$. 
A duality argument then leads to that $\ker(\L_0^\dagger) = \spn\{\delta M,\delta P\}$ and, 
therefore, we deduce from (N3) of Assumption~\ref{A:kerL0} that 
$w_1$ and $w_3$ form a basis of $\ker(\L_0^\dagger)$. 

To proceed, notice that $v_3 \in \ker(\L_0^2)$ by the latter equation in \eqref{E:H2},
but $v_3 \notin \ker(\L_0)$ by (N1) of Assumption~\ref{A:kerL0}. 
If $v \in\ker(\L_0^2)/\ker(\L_0)$ then $\L_0 v = v_1$, 
which in light of (N3) of Assumption~\ref{A:kerL0} admits no solutions 
other than $v_3$ by the Fredholm alternative. A dual statement follows mutatis mutandis, 
noting that $M_a u_c - M_cu_a$ is orthogonal to $\delta M$ and thereby is in ${\rm range}(J)$. 
Similarly $\ker(\L_0^3)/\ker({\mathcal L_0^{2}})$ must be empty 
since $\L_0 v = v_3$ admits no solutions by the Fredholm alternative.
\end{proof}

In case $G=0$, the Jordan block structure for the periodic, generalized null space of $\L_0$ is 
necessarily larger than {(J1)}-{(J3)}. Since $\langle w_j,v_k\rangle=0$ if $j\neq k$, indeed, 
either there must be an element in $\ker(\L_0^\dagger)$ linearly independent of $v_1$ and $v_2$, 
or since $v_j$'s would then lie in $\ker(\L_0^\dagger)^\perp = \ran(\L_0)$, 
there must be elements in $\ker(\L_0^2)$ and $\ker(\L_0^3)$ linearly independent of $v_3$.
Here and elsewhere, the superscript $\perp$ means the orthogonal complement. 

\medskip

In the case of the generalized KdV equation (see \eqref{E:gKdV}), 
whose traveling wave equation reduces by quadrature to that
\begin{equation}\label{E:traveling-gKdV} 
\frac12u_x^2+F(u)+\frac12cu^2+au=b, \qquad \text{where }F'=f,
\end{equation}
for some $b \in \mathbb{R}$, stability indices were effectively calculated 
in \cite{BJ2010}, for instance, in terms of $P$, $M$, $T$ as functions of $c$, $a$, $b$ 
(although the results may be restated in terms of $H$, $P$, $M$). 
When dealing with abstract Hamiltonian systems, however, 
the second constant of integration $b$ may not be available 
and we opt to choose $T$ as a periodic traveling wave parameter instead. 
Accordingly we shall express the modulational instability index 
in terms of $H$, $P$, $M$ as functions of $c$, $a$, $T$. 
The present approach seems to lead to advantages that 
$u_a$, $u_x$, $u_c$ are $T$-periodic, as opposed to in \cite{BJ2010},
and they form a basis of the periodic, generalized null space of $\L_0$, 
in connection to variational properties of the equation. 
Formulae do become cumbersome. 
But in the absence of extra features, this seems the only way forward. 

\subsection{Jordan block perturbation and modulational instability}\label{SS:perturbation}
Let $u(\cdot;c,a,T)$ be a periodic traveling wave of \eqref{E:equation}. 
Under Assumption~\ref{A:HPM}, Assumption~\ref{A:periodic} and Assumption~\ref{A:kerL0}
we shall examine the $L^2_{per}([0,T])$-spectrum of $\L_\tau(u;c,a)$, defined in \eqref{E:eigen-tau},
in the vicinity of the origin for $|\tau|$ small, where
a Baker-Campbell-Hausdorff expansion reveals that 
\begin{equation}\label{E:expL}
 \mathcal{L}_\tau(u;c,a)=L_0(u;c,a)+i \tau L_1(u;c,a)-\frac12 \tau^2  L_2(u;c,a)+\cdots,
\end{equation} 
\begin{equation}\label{D:L12}
L_0(u;c,a)=\L_0(u;c,a), \quad L_1(u;c,a)=[L_0, x],\quad L_2(u;c,a) =[L_1, x], \dots.
\end{equation}
Notice that $L_0$, $L_1$, $L_2,\dots$ are well-defined in $L^2_{per}([0,T])$ even though $x$ is not. 
In the case of the generalized KdV equation (see \eqref{E:gKdV}),
\[
L_0=\partial_x(-\partial_{x}^2-f(u)-c),\quad L_1=-3\partial_{x}^2-f(u)-c, \quad L_2=-6\partial_x,\dots.
\] 

Recall from Lemma~\ref{L:L0} that zero is a generalized $L^2_{per}([0,T])$-eigenvalue of $\L_0$, 
with algebraic multiplicity three and geometric multiplicity two
(with a basis $\{v_j\}$ and a dual basis $\{w_j\}$, $j=1, 2, 3$,
of the generalized $L^2_{per}([0,T])$-null space).
For $|\tau|$ small, therefore, three eigenvalues of $\L_\tau$ will branch off from the origin. 
We make an effort to understand when an eigenvalue of $\L_\tau$ leaves the imaginary axis.

\medskip

Varying $\tau$ in \eqref{E:eigen-tau} for $|\tau|$ small, we substitute \eqref{E:expL}  
to arrive at the perturbation problem
\begin{equation}\label{E:perturbation}
\Big(L_0 +\epsilon L_1 +\frac12 \epsilon^2 L_2+O(\epsilon^3)\Big)\phi (\epsilon) 
=\mu(\epsilon)\phi(\epsilon),
\end{equation}
where $\epsilon,\mu\in\mathbb{C}$ are near $\epsilon=0$, $\mu=0$ and $\phi \in L^2_{per}([0,T])$. 
Note that $\epsilon$ is related to the Floquet exponent via $\epsilon=i\tau$. 

Eigenvalues of \eqref{E:perturbation} in the neighborhood of the origin in general bifurcate 
merely continuously in the perturbation parameter $\epsilon$, but {\em not} in the $C^1$ manner. 
Rather they admit Puiseaux series in fractional powers of $\epsilon$. Requiring that 
\begin{equation}\label{E:w13v2} 
\langle w_1, L_1v_2 \rangle =0\quad\text{and}\quad
\langle w_3, L_1v_2 \rangle =0,\end{equation}
however, eigenvalues do depend upon the perturbation parameter in the $C^1$ manner; 
a proof based upon the Fredholm alternative may be found in \cite[Theorem~4]{BJ2010}. 
In applications in Section~\ref{S:KdV} and Section~\ref{S:gKdV}
we shall in fact demonstrate that 
\begin{equation}\label{E:parity}
\langle w_j, L_\ell v_k\rangle=0 \qquad \text{whenever $j+k+\ell$ is even,}
\end{equation}
where $j,k=1,2,3$ and $\ell=0,1,2$. Hence we may posit that
\begin{equation}\label{D:mu-phi}
\mu(\epsilon) = \epsilon \mu_1 + \epsilon^2 \mu_2 + \cdots \quad \text{and}\quad
\phi(\epsilon) = \phi_0 + \epsilon \phi_1 +  \epsilon^2  \phi_2+\cdots.
\end{equation}
In the case of generalized KdV equations (see \eqref{E:gKdV}), incidentally, 
eigenvalue bifurcation is analytic.

Note from the Fredholm alternative that $L_0 \phi = b$ is solvable if
$b \in {\rm range}(L_0) = \ker(L_0^\dagger)^\perp$ 
and the solution is defined up to an element in $\ker(L_0)$. 
Below we write the solution as $\phi = L_0^{-1} b$, with the understanding that 
${\rm range}(L_0^{-1})\perp \spn\{w_1,w_2\}$. As a reminder, 
\begin{equation}\label{E:kerL0}
\ker(L_0)=\spn\{v_1, v_2\} \quad\text{and}\quad {\rm range}(L_0)=\spn\{w_1, w_3\}^\perp.
\end{equation}

\medskip

Substituting into \eqref{E:perturbation} the eigenvalue and eigenfunction representations 
in \eqref{D:mu-phi}, at the order of $\epsilon^0=1$, we find that
\[
L_0 \phi_0 = 0,
\]
whence $ \phi_0= c_1v_1 + c_2 v_2$ for some $c_1, c_2 \in \mathbb{C}$. 

At the order of $\epsilon$, correspondingly, we find that
\[
L_0 \phi_1 +L_1  \phi_0= \mu_1 \phi_0,
\] 
which by virtue of the Fredholm alternative is solvable if 
\begin{align*}
0=&\langle w_1, (\mu_1-L_1)\phi_0\rangle=
c_1(\mu_1\langle w_1, v_1\rangle-\langle w_1, L_1v_1\rangle)
+c_2(\mu_2\langle w_1,v_2\rangle-\langle w_1, L_1v_2\rangle), \\
0=&\langle w_3, (\mu_1-L_1)\phi_0\rangle=
c_1(\mu_1\langle w_3, v_1\rangle-\langle w_3, L_1v_1\rangle)
+c_2(\mu_2\langle w_3,v_2\rangle-\langle w_3, L_1v_2\rangle).
\end{align*}
Since the last terms on the right sides vanish by \eqref{E:w13v2}, 
these reduce, with the help of \eqref{E:G}, to that 
\[ 
c_1(\langle w_1, L_1v_1\rangle-\mu_1G)=0\quad\text{and}
\quad c_1\langle w_3, L_1v_1\rangle=0.
\]
Therefore, $c_1=0$\footnote{
The kernel of $L_0$ is spanned by two elements while $L_0 + \epsilon L_1$
for $\epsilon$ small but non-zero supports three eigenfunctions at the origin,
which in the limit as $\epsilon \rightarrow 0$ tend to the same limit. 
Numerical experiments bear this out.} and 
\begin{equation}\label{E:phi12} 
\phi_0=c_2v_2,\qquad \phi_1=c_2L_0^{-1}(\mu-L_1)v_2+c_3v_1.
\end{equation}
Since $\phi_1$ is determined merely up to an element in $\ker(L_0)$ 
one must add $c_3 v_1 + c_4 v_2$ to it. Any component in the $v_2$ direction, however, 
may be absorbed to $\phi_0$. Hence, without loss of generality, we set $c_4=0$. 
This amounts to fixing normalization in the perturbation theory for symmetric operators. 
 
Continuing, at the order of $\epsilon^2$, we find that
\[ 
L_0 \phi_2 +  L_1\phi_1 + \frac12 L_2 \phi_0 = \mu_1\phi_1 + \mu_2\phi_0, 
\]
which by the Fredholm alternative is solvable if 
$\l w_1, L_0\phi_2\r=0$ and $\l w_3, L_0\phi_2\r=0$. 
Substituting \eqref{E:phi12} we arrive at that 
\begin{multline*}
c_2\langle w_1, L_0^{-1}(\mu_1-L_1)v_2\rangle
+c_3\langle w_1, L_1\phi_1\rangle+\frac12 c_2\langle w_1, L_2v_2\rangle\\
=c_2\mu_1\langle w_1, L_0^{-1}(\mu_1-L_1)v_2\rangle
+c_3\langle w_1, v_1\rangle+c_2\mu_2\langle w_1, v_2\rangle
\end{multline*}
and
\begin{multline*}
c_2\langle w_3, L_0^{-1}(\mu_1-L_1)v_2\rangle
+c_3\langle w_3, L_1\phi_1\rangle+\frac12 c_2\langle w_3, L_2v_2\rangle\\
=c_2\mu_1\langle w_3, L_0^{-1}(\mu_1-L_1)v_2\rangle
+c_3\langle w_3, v_1\rangle+c_2\mu_2\langle w_3, v_2\rangle.
\end{multline*}
Equivalently
\begin{equation}\label{E:matrix-mu} 
\left(\begin{matrix} 
a_{22}\mu_1^2+b_{22}\mu_1+c_{22} & a_{23}\mu_1+b_{23} \\
a_{32}\mu_1^2+b_{32}\mu_1+c_{32} & a_{33}\mu_1+b_{33} 
\end{matrix}\right)
\left( \begin{matrix} 
c_2 \\ c_3 \end{matrix}\right) =0,
\end{equation}
where, after simplifying various inner products with the help of \eqref{E:G} and \eqref{E:w13v2} 
and noting from  the former equation in \eqref{E:vw3} and the latter equation in \eqref{E:kerL0}
that $L_0^{-1}v_2=w_3$, 
\begin{subequations}
\begin{align}
a_{22}& =-\langle w_1, L_0^{-1}v_2\rangle=-\langle w_1, w_3\rangle=0, \label{D:a22}\\
b_{22}&=\langle w_1, L_1L_0^{-1}v_2\rangle+\langle w_1, L_0^{-1}L_1v_2\rangle
=\langle w_1, L_1v_3\rangle, \label{D:b22}\\
c_{22}&=-\langle w_1, L_1L_0^{-1}L_1v_2\rangle+\frac12\langle w_1, L_2v_2\rangle,\label{D:c22}\\
a_{23}&=-\langle w_1,v_1\rangle=-G, \label{D:a23}\\ 
b_{23}&=\langle w_1, L_1v_1\rangle, \label{D:b23} 
\intertext{and}
a_{32}& =-\langle w_3, L_0^{-1}v_2\rangle=-\langle w_3, w_3\rangle=-G, \label{D:a32}\\
b_{32}&=\langle w_3, L_1L_0^{-1}v_2\rangle+\langle w_3, L_0^{-1}L_1v_2\rangle
=\langle w_3, L_1v_3\rangle+\langle w_2, L_1v_2\rangle, \label{D:b32}\\
c_{32}&=-\langle w_3, L_1L_0^{-1}L_1v_2\rangle+\frac12\langle w_3, L_2v_2\rangle, \label{D:c32}\\
a_{33}&=-\langle w_3,v_1\rangle=0, \label{D:a33}\\ 
b_{33}&=\langle w_3, L_1v_1\rangle. \label{D:b33}
\end{align}
\end{subequations} 

\medskip

Observe that \eqref{E:matrix-mu} is a quadratic eigenvalue problem, 
which may be transformed into a linear matrix pencil of the form $({\bf B} - \mu_1G {\bf A}) \vec{c} = 0$,
after the change of variables $c_2 \mapsto -G^{-1}c_2$ and $c_1 =(a_{32} \mu_1+b_{32}) c_2$,
so long as $G\neq 0$. Consequently \eqref{E:matrix-mu} is equivalent to that
\begin{equation}\label{E:matrix-mu'}
({\bf B} - \mu_1G {\bf A}) \vec{c}:=\left( \left(\begin{matrix} 
0 & -G c_{32} & b_{33} \\
1 & -b_{32}  & 0 \\
0 & - G c_{22}  & b_{23}   
\end{matrix}\right) - \mu_1 G  \left(\begin{matrix} 
1 & 0 & 0 \\
0 & -1 & 0 \\
0 &  b_{22} & 1  
\end{matrix}\right) \right)
\left( \begin{matrix} 
c_1 \\ c_2 \\ c_3 \end{matrix}\right) =0.
\end{equation}
Since ${\bf A}$ is invertible, \eqref{E:matrix-mu}, or \eqref{E:matrix-mu'}, is further equivalent to that
\[
({\bf D} - \mu_1 G{\bf I})\vec{c} =0,
\] 
where 
\begin{equation}\label{D:dispersion}
{\bf D}:= {{\bf A}^{-1}} {\bf B} = \left(\begin{matrix} 
0 & -G c_{32} & b_{33} \\
-1 & b_{32}  & 0 \\
b_{22} & - G c_{22} - b_{22}b_{32}  & b_{23}   
\end{matrix}\right)
\end{equation}
is the {\em effective dispersion matrix}.

\medskip

In view of \eqref{E:perturbation} and \eqref{D:mu-phi}, noting that $\epsilon=i\tau$ and 
$\mu_1$ is an eigenvalue of \eqref{E:matrix-mu}, or equivalently \eqref{E:matrix-mu'}, 
we ultimately obtain spectral curves of \eqref{E:eigen-tau} near the origin for $|\tau|$ small. 

\begin{theorem}[Normal form]
Under Assumption~\ref{A:HPM}, Assumption~\ref{A:periodic}, Assumption~\ref{A:kerL0}
and \eqref{E:w13v2}, three $L^2_{per}([0,T])$-eigenvalues of \eqref{E:eigen-tau} of the form
\begin{equation}\label{E:muj}
\mu_j(\tau) = i G^{-1}\mu_j^0 \tau + O(\tau^2), \qquad j=1, 2, 3,
\end{equation}
bifurcate from zero for $|\tau|$ small, where $G$ is defined in {\rm (N3)} of Assumption \ref{A:kerL0} 
and $\mu_j^0$'s are eigenvalues of ${\bf D}$ in \eqref{D:dispersion}. 
\end{theorem}

Furthermore a complex eigenvalue of ${\bf D}$ implies modulational instability,
and the discriminant of its characteristic polynomial, or equivalently $\det({\bf D}-\mu G{\bf I})$, 
leads to a modulational instability index. A straightforward calculation reveals that
\begin{equation}\label{D:P}
\det({\bf D}-\mu G{\bf I}):=-G^3\mu^3+D_2\mu^2+D_1\mu+D_0,
\end{equation}
where 
\begin{subequations}
\begin{align}
D_2&=G^2(b_{23}+b_{32})=
G^2(\l w_1,L_1v_1\r+\l w_2,L_1v_2\r +\l w_3,L_1v_3\r ),\label{D:B} \\
D_1&=G(b_{22}b_{33}-b_{23}b_{32}+Gc_{32}),\label{D:C} \\
D_0&=G(c_{22}b_{33}-c_{32}b_{23}).\label{D:D}
\end{align}
\end{subequations}

We summarize the conclusion.

\begin{corollary}[Modulational instability index]\label{T:instability} 
Under Assumption~\ref{A:HPM}, Assumption~\ref{A:periodic}, Assumption~\ref{A:kerL0} 
and \eqref{E:w13v2} a periodic traveling wave $u(\cdot\,; c,a,T)$ of \eqref{E:equation} 
is unstable to long wavelengths perturbations 
if $\det({\bf D}-\mu G{\bf I})$ admits a complex root, or equivalently, if its discriminant 
\begin{equation}\label{D:index}
\Delta:=D_2^2D_1^2+4G^3D_1^3-4D_2^3D_0-27G^6D_0^2-18G^3D_2D_1D_0
\end{equation}
is negative, where $G$ is defined in {\rm (N3)} of Assumption \ref{A:kerL0} and 
$D_2$, $D_1$, $D_0$ are in \eqref{D:B}-\eqref{D:D}
and \eqref{D:a22}-\eqref{D:b23}, \eqref{D:a32}-\eqref{D:b33}.
\end{corollary}

We therefore obtain the {\em modulational instability index} $\Delta$ in terms of various inner products 
between basis and dual basis elements of the generalized null space of $L_0$,
together with $L_0^{-1}$, $L_1$, $L_2$. We may further express the index 
in terms of $H$, $P$, $M$ and their derivatives with respect to $c$, $a$, $T$
with the help of variational properties of the equation.

\medskip

We remark that while ${\bf D}$ is made up of terms up to second order in various inner products, 
its characteristic polynomial is homogeneous. In fact, $\tr({\bf D})$ is linear in inner products, 
$\tr({\bf D}^2)$ is quadratic and $\det({\bf D})$ is cubic.

\subsection{Pohozaev identity techniques}\label{SS:pohozaev}
One may calculate various inner products in \eqref{D:a22}-\eqref{D:b23} and 
\eqref{D:a32}-\eqref{D:b33} using definitions in \eqref{D:vw1}-\eqref{D:vw3} and \eqref{D:L12}, 
except $\l w_1, L_1L_0^{-1}L_1 v_2\r$ and $\l w_3, L_1L_0^{-1}L_1 v_2\r$ 
in \eqref{D:c22} and \eqref{D:c32}, respectively. 
The goal of this subsection is to develop Pohozaev type identities,
which assist us in calculating them without recourse to inversion of $L_0$.

\begin{lemma}[Periodic Pohozaev-type identity]\label{L:pohozaev1}
If $u$ is $T$-periodic and satisfies \eqref{E:traveling}, or equivalently \eqref{E:traveling'}, then 
$x u_x+T u_T$ is $T$-periodic and satisfies
\begin{equation}\label{E:pohozaev}
 L_0(xu_x+Tu_T) =L_1u_x, 
\end{equation}
where $L_0$ and $L_1$ are in \eqref{D:L12}.
If in addition $u$ is even and satisfies Assumption~\ref{A:kerL0} then
\begin{equation}\label{E:pohozaev2}
L_0^{-1} L_1 u_x = xu_x+Tu_T - G^{-1}\langle w_1, xu_x+Tu_T\rangle v_1,  
\end{equation}
where $v_1$, $w_1$ and $G$ are, respectively, in \eqref{D:vw1} and 
{\rm (N3)} of Assumption~\ref{A:kerL0}.
\end{lemma}

\begin{proof}
Since $L_0u_x=\L_0u_x=0$ (see \eqref{D:vw2} and \eqref{E:vw2}), one may write that
\[
L_1u_x := [L_0,x]u_x= L_0(xu_x) - xL_0u_x = L_0(xu_x)
\] 
formally in the non-periodic functions setting. 
Unfortunately we must modify it in the periodic functions setting. 
For one thing, although $L_0x-x L_0$ is well-defined in the periodic setting, 
$L_0x$ and $xL_0$ individually are {\em not}. Another, related, is that 
$xu_x$ is {\em not} $T$-periodic, but it develops a jump in the derivative over one period:
\[ 
[x u_x ]_0^T = 0 \quad\text{and}\quad [(x u_x)_x]_0^T = T u_{xx}(T;c,a,T).
\]
On the other hand, this is what causes $u_T$ to fail to lie in the periodic kernel of~$\delta^2 E$;
see Remark~\ref{R:uT}. Indeed $\delta^2 E u_T = 0$, 
acting on smooth functions, although $u_T$ is {\em not} $T$-periodic. Rather
\[ 
[u_T ]_0^T = 0 \quad\text{and}\quad [(u_T)_x]_0^T = -u_{xx}(T;a,c,T).
\]
We then observe that the jump in the derivative of $Tu_T$ offsets that in $x u_x$, 
so that $x u_x + T u_T$ makes a $T$-periodic function. Therefore \eqref{E:pohozaev} follows.
Note in passing that the vector field $x \partial_x + T \partial_T$ corresponds to 
simultaneous rescaling of the spatial variable and the period, maintaining periodicity. 

Furthermore, since $L_0^{-1}$ is defined up to an element in $\ker(L_0)=\text{span}\{v_1,v_2\}$
(see \eqref{E:kerL0}), one may write that
\[
L_0^{-1}L_1v_2=xu_x+Tu_T+c_1v_1+c_2v_2
\] 
for some $c_1, c_2$ constants. Since $u$ is even, a parity argument dictates that $c_2=0$. 
Since $\ran(L_0^{-1}) \perp \{w_1,w_2\}$ (see \eqref{E:kerL0}), moreover, \eqref{E:pohozaev2} follows 
after taking the inner product against $w_1$ and noting \eqref{E:G}.  
\end{proof}
 
Here we tacitly assume the continuity of $xu_x+Tu_T$ and $(xu_x+Tu_T)_x$ across the period 
so that $xu_x+Tu_T$ lies in the domain of the Hamiltonian. 
We shall establish in Proposition~\ref{P:existence} that a periodic traveling wave of 
a KdV equation with fractional dispersion is in fact smooth if it arises as a local constrained minimizer.

\medskip

The apparent lack of an identity like \eqref{E:pohozaev} is 
the main obstruction in extending the present development to higher dimensions. 

\medskip

Concluding the subsection we discuss another Pohozaev identity,
relating inner products involving $xu_x+Tu_T$ to derivatives with respect to the wave number. 

\begin{lemma}\label{L:pohozaev2}
If $f$ is smooth and $F'=f$ then
\begin{align}
\langle f(u), xu_x + Tu_T\rangle
=& -\int^T_0 F(u)~dx+T\Big(\int^T_0 F(u)~dx \Big)_T \label{E:Fphi'}  \\
=& -\Big(\Omega \int^T_0 F(u)~dx \Big)_\Omega, \label{E:Fphi}
\end{align}
where $\Omega=1/T$ denotes the wave number. 
\end{lemma}

\begin{proof}
After integration by parts, 
\begin{align*}
\langle f(u), xu_x + Tu_T\rangle&= \int_0^T (x u_x f(u(x)) + T u_T f(u(x)))~dx  \\
&= T F(u(T)) - \int_0^T F(u(x))~dx + T \int_0^T u_T f(u(x))~dx.
\end{align*}
Since
\[\Big( \int_0^T F(u(x))~dx\Big)_T = F(u(T)) + \int_0^T u_T f(u(x))~dx,\]
it follows that
\begin{align*}
\langle f(u), xu_x + Tu_T\rangle &= T \Big(\int_0^T F(u)~dx\Big)_T - \int_0^T F(u)~dx \\
&= T^2\Big(\frac{1}{T}\int_0^T F(u)~dx\Big)_T=- \Big(\Omega \int_0^T F(u)~dx\Big)_\Omega.
\end{align*}
\end{proof}

\section{Application: KdV equations with fractional dispersion}\label{S:KdV}
We shall illustrate the results in Section \ref{S:theory} by discussing the KdV equation 
with fractional dispersion and the quadratic power-law nonlinearity
\begin{equation}\label{E:KdV}
u_t-\Lambda^\alpha u_x+(u^2)_x=0, 
\end{equation}
where $0<\alpha\leq 2$ and $\Lambda=\sqrt{-\partial_x^2}$ is defined via the Fourier transform 
as $\widehat{\Lambda u}(\xi)=|\xi|\hat{u}(\xi)$. In the range $0<\alpha<1$, alternatively, 
\[
\Lambda^\alpha u(x)=C(\alpha)\,PV\int_{-\infty}^\infty \frac{u(x)-u(y)}{|x-y|^{1+\alpha}}~dy,
\]
where $PV$ stands for the Cauchy principal value and 
$C(\alpha)$ is a normalization constant. 

In the case of $\alpha=2$, notably, \eqref{E:KdV} recovers the KdV equation
while in the case of $\alpha=1$ it corresponds to the Benjamin-Ono equation. 
In the case\footnote{
Note that $\Lambda^\alpha\partial_x$ is not singular for $\alpha\geq -1$.} 
of $\alpha=-1/2$, furthermore, it was argued in \cite{Hur-breaking} to approximate 
up to ``quadratic" order the water wave problem in two spatial dimensions in the infinite depth case. 
Notice that \eqref{E:KdV} is nonlocal for $0<\alpha<2$. 
Incidentally fractional powers of the Laplacian occur in numerous applications, 
such as dislocation dynamics in crystals (see \cite{CFM}, for instance) 
and financial mathematics (see \cite{CT}, for instance). 

The present treatment may be adapted mutatis mutandis to general power-law nonlinearities; 
see Remark \ref{R:gKdV}. We focus on the quadratic nonlinearity, however, to simplify the exposition. 
Incidentally it is characteristic of many wave phenomena; see \cite{Whitham}, for instance. 

\medskip

Throughout the section and the followings we'll work in the periodic, $L^2$-Sobolev spaces setting 
over the interval $[0,T]$, where $T>0$ is fixed although at times it is treated as a free parameter. 
We define a periodic Sobolev space of fractional order via the norm 
\[ 
\|u\|_{H^{\alpha/2}_{per}([0,T])}^2=\int^T_0 (u^2+u\Lambda^\alpha u)~dx, \qquad 0<\alpha<2.
\]
We use $\langle \, , \rangle$ for the $L^2_{per}([0,T])$-inner product. 

\subsection{Periodic traveling waves}\label{SS:periodicBO}
Notice that \eqref{E:KdV} may be written in the Hamiltonian form \eqref{E:equation}, 
for which $J=\partial_x$ and
\begin{equation}\label{E:HKdV}
H(u)=K(u)+U(u),\end{equation}
where
\begin{equation}\label{D:KU}
K(u)=\int^T_0 \frac12 u\Lambda^\alpha u~dx\quad\text{and}\quad U(u)= \int^T_0- \frac13u^3~dx
\end{equation}
correspond, respectively, to the kinetic and potential energies.
Notice that \eqref{E:KdV} possesses, in addition to $H$, two conserved quantities
\begin{equation}\label{E:PKdV}
P(u)=\int^T_0 \frac12 u^2~dx 
\end{equation}
and
\begin{equation}\label{E:MKdV}
M(u)=\int^T_0 u~dx,
\end{equation}
which correspond, respectively, to the momentum and the mass. 
Clearly $H$, $P$, $M$ are smooth in $H^{\alpha}_{per}([0,T])\cap L^3_{per}([0,T])$. Since 
\begin{equation}\label{E:PM}
\delta P(u)=u\quad\text{and}\quad\delta M(u)=1,
\end{equation}
moreover, $H$, $P$, $M$ satisfy Assumption~\ref{A:HPM}. 
Incidentally \eqref{E:KdV} is invariant under 
\begin{equation}\label{E:scaling} 
u(t,x) \mapsto \lambda^\alpha u(\lambda (x+x_0), \lambda^{\alpha+1}t)\end{equation}
for any $\lambda>0$ for any $x_0 \in \mathbb{R}$. 

\medskip

A periodic traveling wave of \eqref{E:KdV} takes the form $u(x,t)=u(x+x_0+ct)$, 
where $c\in \mathbb{R}$, $x_0 \in \mathbb{R}$ and 
$u$ is $T$-periodic, satisfying by quadrature that
\begin{equation}\label{E:pKdV}
\Lambda^\alpha u-u^2-cu-a=0
\end{equation}
for some $a \in \mathbb{R}$ (in the sense of distributions), or equivalently, 
\begin{equation}\label{E:pKdV'}
\delta E(u;c,a):=\delta(H(u)-cP(u)-aM(u))=0. 
\end{equation}
Henceforth we shall write a periodic traveling wave of \eqref{E:KdV} as $u=u(\cdot\,;c,a)$,
unless specified otherwise. In a more comprehensive description, it depends upon four parameters 
$c$, $a$ and $T$, $x_0$. Note, however, that $T>0$ is arbitary. 
Corresponding to translational invariance (see \eqref{E:scaling}), moreover, 
$x_0$ plays no significant role. Hence we may mod it out.

A solitary wave, whose profile vanishes asymptotically, corresponds to $a=0$ and $T=+\infty$.

\medskip

Clearly a periodic traveling wave of \eqref{E:KdV} satisfies 
\eqref{E:H1}-\eqref{E:H3} and \eqref{E:PaMc}. Below we develop integral identities 
that a periodic solution of \eqref{E:pKdV}, or equivalently \eqref{E:pKdV'}, a priori satisfies 
and which will be useful in various proofs.

\begin{lemma}[Integral identities]\label{L:identities}
If $u \in H^{\alpha/2}_{per}([0,T]) \cap L^3_{per}([0,T])$ satisfies \eqref{E:pKdV} then 
\begin{align}
2P+cM+&aT=0, \label{E:I1-BO} \\ 
2K+3U-2cP&-aM=0, \label{E:I3-BO} \\
(\alpha+1)K\,+2U-&cP+TE_T=0. \label{E:I2-BO}
\end{align}
\end{lemma}

\begin{proof}
Integrating \eqref{E:pKdV} over the periodic interval $[0,T]$ manifests \eqref{E:I1-BO}. 

Multiplying \eqref{E:pKdV} by $u$ and integrating over $[0,T]$ lead to \eqref{E:I3-BO}.
Multiplying it by $xu_x+Tu_T$ and integrating over $[0,T]$ moreover lead, 
with the help of \eqref{E:Fphi'}, to that 
\begin{equation}\label{E:I2'-BO}
(\alpha-1)K-U+cP+aM+TE_T=0.
\end{equation}
Indeed, since $[\partial_x\Lambda^\alpha, x]=(\alpha+1)\Lambda^\alpha$ by brutal force, 
an integration by parts reveals that
\begin{align*}
\int^T_0 xu_x\Lambda^\alpha u~dx=\frac{\alpha-1}{2}\int^T_0 u\Lambda^\alpha u~dx.
\end{align*}
Adding \eqref{E:I3-BO} and \eqref{E:I2'-BO} then proves \eqref{E:I2-BO}.
Incidentally one may write in light of Lemma~\ref{L:pohozaev2} that 
\begin{equation}\label{E:Kphi}
\l\Lambda^\alpha u,xu_x+Tu_T\r =\alpha K-(\Omega K)_\Omega,
\end{equation}
where $\Omega=1/T$.
\end{proof}

If $u\in H^{\alpha/2}(\mathbb{R})\cap L^3(\mathbb{R})$ is a solitary wave of \eqref{E:KdV} then 
\eqref{E:I1-BO}-\eqref{E:I2-BO} reduce to
\begin{equation}\label{E:I-solitary}
2P+cM=0,\quad 2K+3U-2cP=0,\quad (\alpha+1)K+2U-cP=0,
\end{equation}
respectively.

\medskip

In the case of $\alpha=2$, periodic traveling waves of \eqref{E:KdV}, namely the KdV equation, 
are known in closed form and they go by the name of cnoidal waves (see \cite{KdV}, for instance). 
In the case of $\alpha=1$, Benjamin \cite{Benjamin} exploited the Poisson summation formula 
and derived an explicit form of periodic traveling waves of \eqref{E:KdV}:
\begin{equation}\label{E:pBO}
u(x;c,a,T)=\frac{2\pi}{T}\frac{{\displaystyle \frac{\frac{2\pi}{T}}{\sqrt{c^2-4a-(\frac{2\pi}{T})^2}} }}
{{\displaystyle \sqrt{\frac{c^2-4a}{c^2-4a-(\frac{2\pi}{T})^2}} }
-\cos\Big({\displaystyle \frac{2\pi x}{T}}\Big)}-\frac12(\sqrt{c^2-4a}+c),
\end{equation}
where $c<0$\footnote{Thanks to Galilean invariance, $u(\cdot\,;-c,a,T)+c$, $c>0$, 
is a periodic traveling wave of \eqref{E:KdV}, $\alpha=1$, as well; see Section~\ref{SS:BO}. 
Incidentally Benjamin's derivation in \cite{Benjamin} requires $c<0$ 
to ensure that the infinite sum of certain Fourier coefficients converge.}, 
$c^2-4a-(\frac{2\pi}{T})^2>0$ and $T>0$ is arbitrary.
In general, the existence of periodic traveling waves of \eqref{E:KdV} may follow from 
variational arguments, although one may lose an explicit form of the solution. 
In the energy subcritical case, i.e., $\alpha>1/3$, in particular, 
a family of periodic traveling waves of \eqref{E:KdV}  locally minimizes the Hamiltonian 
subject to conservations of the momentum and the mass, 
analogously to ground states in the solitary wave setting. 

\begin{proposition}[Existence, symmetry and regularity]\label{P:existence}
Let $1/3<\alpha\leq 2$. A local minimizer $u \in H^{\alpha/2}_{per}([0,T])$ for $H$
subject to that $P$ and $M$ are conserved exists for each $0<T<\infty$. 
It satisfies \eqref{E:pKdV} for some $c\neq 0$ and $a \in \mathbb{R}$, 
and it depends upon $c$ and $a$ in the $C^1$ manner.

Moreover $u$ is even and strictly decreases over the interval $[0,T/2]$, 
$u \in H^\infty_{per}([0,T])$. 
\end{proposition}

To interpret, a local energy minimizer for \eqref{E:pKdV} subject to 
conservations of the momentum and the mass satisfies Assumption~\ref{A:periodic}.
 
\begin{proof}
It suffices to take $c=-1$ and $a=0$. Suppose $a\neq 0$; 
we may then assume that $c$ and $M$ are of opposite sign and $a> 0$. 
For, in case $c$ and $M$ are of the same sign, noting that \eqref{E:KdV} is time reversible, 
$t \mapsto -t$ in \eqref{E:KdV} reverses the sign of $c$ in \eqref{E:pKdV} 
while leaving other components of the equation invariant. 
Once $c$ and $M$ are of opposite sign, $a\geq 0$ follows 
from \eqref{E:I1-BO} since $P\geq 0$ and $T>0$. 
We then devise the change of variables $u \mapsto u+\frac12 (\sqrt{c^2+4a}-c)$ 
and \eqref{E:pKdV} becomes
\begin{equation}\label{E:a=0} 
\Lambda^\alpha u-u^2+\gamma u=0,\qquad \text{where $\gamma =\sqrt{c^2+4a}>0$.}
\end{equation}
Incidentally it is reminiscent of that \eqref{E:KdV} obeys Galilean invariance 
under $u(x,t)\mapsto u(x, t)+u_0$ for any $u_0 \in \mathbb{R}$. 
Thanks to scaling invariance (see \eqref{E:scaling}) we further devise the change of variables 
$u(x)\mapsto 1/\gamma\, u(x/\gamma^{\alpha})$ and \eqref{E:a=0} becomes
\begin{equation}\label{E:a=0,c=-1}
\Lambda^\alpha u-u^2+ u=0.
\end{equation}
To recapitulate, we may take $c=-1$, $a=0$ and seek a (local) minimizer for $H+P$. 

\medskip

Since $H^{\alpha/2}_{per}([0,T])$ in the range $\alpha>1/3$ is compactly embedded 
in $L^3_{per}([0,T])$ by a Sobolev inequality, it is standard from calculus of variations that 
the constrained minimization problem with parameter (abusing notation) $U<0$
\[
K+P=\inf \big\{K(\phi)+P(\phi): \phi \in H^{\alpha/2}_{per}([0,T]),\,U(\phi)=U \big\} 
\]
is attained, say, at $u\in H^{\alpha/2}_{per}([0,T])$. Furthermore it satisfies
\[ \Lambda^\alpha u+ u=\theta u^2\]
for some $\theta\neq0$ in the sense of distributions. We choose $U$ so that $\theta=1$, 
whence $u$ satisfies \eqref{E:a=0,c=-1}. Note from \eqref{E:I3-BO} that  $2K(u)+3U(u)+2P(u)=0$.

Moreover, the constrained minimization problem
\[
E=\inf\{ H(\phi)+P(\phi): \phi\in H^{\alpha/2}_{per}([0,T]),\,
\phi\not\equiv 0,\, 2K(\phi)+3U(\phi)+2P(\phi)=0\}
\]
is attained at $u\in H^{\alpha/2}_{per}([0,T])$. Details are found in \cite[Proposition~2.1]{HJ1}, 
but we merely pause to remark that 
\[
H(\phi)+P(\phi)=K(\phi)+U(\phi)+P(\phi)=\frac13(K(\phi)+P(\phi))
\]
whenever $2K(\phi)+3U(\phi)+2P(\phi)=0$. Since 
\[
\left\langle \delta H(\phi)+\delta P(\phi), \phi\right\rangle=2K(\phi)+3U(\phi)+2P(\phi)
\]
for any $\phi\in H^{\alpha/2}_{per}([0,T])$, furthermore, $u$ minimizes $H+P$ among its critical points.
The existence assertion therefore follows. 

\medskip

To proceed, since the symmetric decreasing rearrangement of $u$ strictly decreases 
$\int^T_0 u\Lambda^\alpha u~dx$, $0<\alpha\leq 2$ while leaving $\int^T_0 u^3~dx$ invariant, 
it follows from rearrangement arguments that a local minimizer for $H$ 
subject to conservations of $P$ and $M$ symmetrically decreases away from 
the principal elevation. The symmetry and monotonicity assertion therefore follows. 
(Notice that unlike in the solitary wave setting, for which $a=0$ and $T=+\infty$, 
a periodic, local constrained minimizer needs not be positive everywhere.)

\medskip

Lastly we address the smoothness of a periodic solution of \eqref{E:pKdV}, or equivalently, 
\begin{equation}\label{E:integralE} 
u=(\Lambda^\alpha+1)^{-1}u^2
\end{equation}
after reduction to $a=0$, $c=-1$ after inversion. We claim that 
if $u \in H^{\alpha/2}_{per}([0,T])$ satisfies \eqref{E:integralE} then $u \in L^\infty_{per}([0,T])$. 
In the case of $\alpha>1$ it follows from a Sobolev inequality, 
while in the case of $1/3<\alpha\leq 1$ a proof based upon 
bounds for the resolvent $(\Lambda^\alpha+1)^{-1}$ is found, for instance, 
in \cite[Lemma A.3]{FL}, albeit in the solitary wave setting. 
We then promote $u \in H^{\alpha/2}_{per}([0,T])\cap L^\infty_{per}([0,T])$ to $H^\alpha_{per}([0,T])$ 
since 
\[
\|\Lambda^\alpha u\|_{L^2}=\Big\|\frac{\Lambda^\alpha}{\Lambda^\alpha+1}u^2\Big\|_{L^2}
\leq \|u^2\|_{L^2} \leq \|u\|_{L^\infty}\|u\|_{L^2}<\infty.
\]
Furthermore a fractional product rule leads to that
\[ 
\|\Lambda^{2\alpha} u\|_{L^2}=\Big\|\frac{\Lambda^{2\alpha}}{\Lambda^\alpha+1}u^2\Big\|_{L^2}
\leq \|\Lambda^\alpha u^2\|_{L^2}\leq C\|u\|_{L^\infty}\|\Lambda^\alpha u\|_{L^2}<\infty
\]
for $C>0$ a constant. After iterating \eqref{E:integralE}, therefore, $u\in H^\infty_{per}([0,T])$. 
\end{proof}

\begin{remark}[Power-law nonlinearities]\label{R:gKdV}\rm
One may rerun the argument in the proof of Proposition~\ref{P:existence} 
in the case of the general power-law nonlinearity
\begin{equation}\label{E:KdV'}
u_t -\Lambda^\alpha u_x+(u^{p+1})_x=0
\end{equation}
and obtain a periodic traveling wave, where $0<\alpha\leq 2$ and $0<p<p_{max}$ 
is an integer such that  
\begin{equation}\label{D:pmax}
p_{max}:=\begin{cases} 
\frac{2\alpha}{1-\alpha} &\text{if $\alpha<1$,} \\ 
+\infty &\text{if $\alpha\geq 1$.}
\end{cases}\end{equation}
In fact, it locally minimizes in $H^{\alpha/2}_{per}([0,T])$ the Hamiltonian
\[ 
\int^T_0 \Big(\frac12 u\Lambda^\alpha u-\frac{1}{p+2}u^{p+2}\Big)~dx
\] 
subject to that $P$ and $M$, defined in \eqref{E:PKdV} and \eqref{E:MKdV}, respectively, 
are conserved. Notice that $0<p<p_{max}$ ensures that
$H^{\alpha/2}_{per}([0,T]) \subset L^{p+2}_{per}([0,T])$ compactly. 
In case $p=1$, it is equivalent to that $\alpha>1/3$. 
\end{remark}

\begin{remark}[Periodic vs. solitary waves]\label{R:solitary}\rm
In the non-periodic functions setting, Weinstein \cite{Weinstein1987} (see also \cite{FL}) proved that 
\eqref{E:pKdV} in the range $\alpha>1/3$ admits a solitary wave, for which $a=0$ and $T=+\infty$. 
In case $\alpha>1/2$ so that \eqref{E:pKdV} is $L^2$-subcritical, in addition, 
the solitary wave further arises as a minimizer for the Hamiltonian subject to constant momentum. 
Periodic, local constrained minimizers, whose existence follows from Proposition \ref{P:existence}, 
are then expected to tend to the solitary wave as their period increases to infinity. 
In case $1/3<\alpha<1/2$, on the other hand, local constrained minimizers for \eqref{E:pKdV} exist
in the periodic wave setting, but they are unlikely to achieve a limiting state 
with bounded energy (the $H^{\alpha/2}$-norm) at the solitary wave limit. 
\end{remark}

One is able to obtain periodic traveling waves of \eqref{E:KdV} for $\alpha\geq -1$, 
with small amplitudes, via perturbative arguments, e.g., local bifurcation theory. 
In the solitary wave setting, in contrast, Pohozaev identities techniques dictate that 
\eqref{E:pKdV} ($a=0$) in the range $\alpha\leq1/3$ does not admit any nontrivial solutions 
in $H^{\alpha/2}(\mathbb{R})\cap L^3(\mathbb{R})$. 

\subsection{Nondegeneracy}\label{SS:kernel}
Throughout the subsection let $u=u(\cdot\,; c, a,T)$ be a periodic traveling wave of \eqref{E:KdV}, 
whose existence follows from Proposition \ref{P:existence}. 
We shall  discuss Assumption~\ref{A:kerL0}. 

\medskip

Clearly $u$ satisfies (N1) of Assumption \ref{A:kerL0}. 

\begin{proposition}[Nondegeneracy of the linearization]\label{P:kernel}
If $u=u(\cdot\,;c,a) \in H^{\alpha/2}_{per}([0,T])$, $0<\alpha\leq 2$, locally minimizes $H$ 
subject to that $P$ and $M$ are conserved for some $c\neq 0$, $a\in\mathbb{R}$ and $T>0$
then the associated linearized operator 
\begin{equation}\label{E:L+KdV}
\delta^2E(u; c,a)=\Lambda^\alpha-2u-c
\end{equation}
acting on $L^2_{per}([0,T])$ is non\-de\-gen\-er\-ate; that is to say,
\[
\ker(\delta^2E(u; c,a))={\rm span}\{u_x\}.
\]
\end{proposition}

The nondegeneracy of the linearization is of fundamental importance 
in the stability of traveling waves and the blowup for the related, time dependent equation; 
see \cite{Weinstein1987, Lin-JFA, KMR} among others. 
But to establish the property is far from being trivial, though. 
Indeed one may cook up a polynomial nonlinearity, say, $f$ so that 
the kernel of $-\partial_x^2-f'(u)$ at the underlying wave is two dimensional at isolated points. 

In the case of generalized KdV equations (see \eqref{E:gKdV}), 
the nondegeneracy of the linearization at a periodic traveling wave was shown 
in \cite{BJ2010}, for instance, 
to be equivalent to that the wave amplitude not be a critical point of the period, 
using the Sturm-Liouville theory for ODEs, 
and it was likewise verified in \cite{Kwong}, among others, at ground states. 
Amick and Toland \cite{AT} demonstrated the property for the Benjamin-Ono equation, 
both in the periodic and solitary wave settings, relating 
the nonlocal, traveling wave equation to a fully nonlinear ODE via complex analysis techniques; 
unfortunately, their arguments are extremely specific to the Benjamin-Ono equation. 
Angulo Pava and Natali \cite{A-PN} made an alternative proof based upon 
the theory of totally positive operators, which however necessitates an explicit form of the solution. 
A satisfactory understanding of the nondegeneracy therefore seems 
largely missing in the case of nonlocal equations. The main obstruction is that 
shooting arguments and other ODE methods may not be applicable to nonlocal operators.

Nevertheless, Frank and Lenzmann \cite{FL} recently demonstrated the property 
at ground states for a class of nonlinear nonlocal equations with fractional Laplaicans. 
Their idea lies in to find a suitable substitute for the Sturm-Liouville theory to count 
the number of sign changes in eigenfunctions of the associated linearized operator. 
Our proof of Proposition \ref{P:kernel} follows along the same line as the arguments 
in \cite[Section~3]{FL}, but with appropriate modifications 
to accommodate the periodic nature of the problem. 

\begin{lemma}[Oscillation of eigenfunctions]\label{L:nodal}
Under the hypothesis of Proposition~\ref{P:kernel} an eigenfunction in 
$H^{\alpha/2}_{per}([0,T]) \cap C^0_{per}([0,T])$ corresponding to the $j$-th eigenvalue of 
$\delta^2E$ changes its sign at most $2(j-1)$ times over the periodic interval $[0,T]$.
\end{lemma}

A thorough proof of Lemma~\ref{L:nodal} may be found in \cite{FL},
albeit in the solitary wave setting. Here we merely hit the main points. 

Notice that $\Lambda^\alpha$, $0<\alpha<2$, may be viewed as the Dirichlet-to-Neumann operator 
for an appropriate (local) elliptic, boundary value problem 
set in the periodic half strip $[0,T]_{per} \times [0,\infty)$. 
Specifically
\[
C(\alpha)\Lambda^\alpha u:=\lim_{y\to 0+}y^{1-\alpha}\phi_y(\cdot, y),
\] 
where $\phi$ solves  
\[ 
\Delta \phi+\frac{1-\alpha}{y}\phi_y=0\quad \text{in $[0,T]_{per} \times (0,\infty)$}, 
\qquad \phi=u \quad\text{on $[0,T]_{per} \times \{0\}$}
\]
and $C(\alpha)$ is a normalization constant. Accordingly one may characterize 
(eigenvalues and) eigenfunctions of \eqref{E:L+KdV} through the Dirichlet type functional 
\[ 
\iint_{[0,T]_{per}\times(0,\infty)}|\nabla\phi(x,y)|^2y^{1-\alpha}~dxdy+\int_0^T(-2u(x)-c)|\phi(x,0)|^2~dx
\]
in a suitable function class. Lemma~\ref{L:nodal} then follows from nodal domain bounds a la Courant. 

\medskip

Below we gather some, mostly elementary, facts about $\delta^2E$. 

\begin{lemma}[Properties of $\delta^2E$]\label{L:L+}
Under the hypothesis of Proposition \ref{P:kernel} the followings hold:
\begin{itemize}
\item[(L1)] $u_x \in \ker(\delta^2E)$ and it corresponds to the lowest eigenvalue of $\delta^2E$ 
restricted to the sector of odd functions in $L^2_{per}([0,T])$;
\item[(L2)] $n_-(\delta^2E)\leq 2$, where $n_-(\delta^2E)$ denotes the number of negative eigenvalues 
of $\delta^2E$ acting on $L^2_{per}([0,T])$, namely the Morse index;
\item[(L3)] $1,u, u^2\in {\rm range}(\delta^2E)$.
\end{itemize} 
\end{lemma}

\begin{proof} 
Differentiating \eqref{E:pKdV} implies that $\delta^2Eu_x=0$. 
Proposition~\ref{P:existence} implies that $u$ may be chosen so that $u_x(x)<0$ for $0<x<T/2$. 
The lowest eigenvalue of $\delta^2E$ acting on $L^2_{per, odd}([0,T])$, 
on the other hand,  must be simple and 
a corresponding eigenfunction is strictly positive (or negative) over the half interval $[0,T/2]$. 
Therefore zero is the lowest eigenvalue of $\delta^2E$ in $L^2_{per, odd}([0,T])$ and 
$u_x$ is a corresponding eigenfunction. 

\medskip

To proceed, since $u$ locally minimizes $H$, and in turn $E$, 
subject to conservations of $P$ and $M$, necessarily, 
\begin{equation}\label{E:2negative} 
\delta^2E|_{\{\delta P(u), \delta M(u)\}^\perp} \geq 0.
\end{equation}
This implies by Courant's mini-max principle that 
$\delta^2E$ admits at most two negative eigenvalues, implying (L2). 
(Unlike in the solitary wave setting, where $n_-(\delta^2E)=1$ at a ground state,
$\delta^2E$ may have up to two negative directions in the periodic wave setting. 
We shall discuss this in Remark~\ref{R:index} below.)

\medskip

Lastly, \eqref{E:H3}, \eqref{E:H2} and \eqref{E:PM} imply that $1, u \in {\rm range}(\delta^2E)$. 
Since 
\[
\delta^2Eu=\Lambda^2u-2u^2-cu=-u^2+a
\] 
by \eqref{E:pKdV}, moreover, $u^2 \in {\rm range}(\delta^2E)$.
\end{proof}

\begin{proof}[Proof of Proposition \ref{P:kernel}]
Considering 
\[
L^2_{per}([0,T])=L^2_{per, odd}([0,T])\oplus L^2_{per, even}([0,T]),
\] 
since $u$ may be chosen to be even by Proposition \ref{P:existence}, we find that 
$L^2_{per, odd}([0,T])$ and $L^2_{per, even}([0,T])$ are invariant subspaces of $\delta^2E$. 
Since 
\[
\ker(\delta^2E|_{L^2_{per, odd}([0,T])})=\text{span}\{u_x\}
\]
by (L1) of Lemma~\ref{L:L+},
it remains to show that $\ker(\delta^2E|_{L^2_{per, even}([0,T])})=\{0\}$. 

Suppose on the contrary that 
there were $\phi \in L^2_{per, even}([0,T])$, $\phi\not\equiv0$, such that $\delta^2E\phi=0$. 
Since $\delta^2E$ has at most two negative eigenvalues by (L2) of Lemma~\ref{L:L+}, 
$\phi$ changes its sign at most twice over the half interval $[0,T/2]$ by Lemma \ref{L:nodal}. 
Consequently, unless $\phi$ is positive (or negative) throughout $[0,T]$, 
either there exists $T_1\in(0,T/2)$ such that $\phi$ is positive (or negative) for $0<|x|<T_1$ 
and negative (or positive, respectively) for $x \in (-T/2,T_1)\cup(T_1,T/2)$, 
or there exist $T_1<T_2$ in $[0, T/2)$ 
such that $\phi$ is positive (or negative) for $|x|<T_1$ and $T_2<|x|<T/2$ 
and $\phi$ is negative (or positive, respectively) for $x\in (-T_2,-T_1)\cup(T_1,T_2)$. 

Since $\phi$ lies in the kernel of $\delta^2E$, on the other hand, it is orthogonal 
to $\text{range}(\delta^2E)$ and, in turn, to $\text{span}\{1,u, u^2\}$ by (L3) of Lemma~\ref{L:L+}. 
In particular $\langle \phi, 1\rangle=0$. Hence $\phi$ cannot be, say, positive throughout $[0,T]$. 
In case $\phi$ is positive for $0<|x|<T_1$ and negative for $x \in (-T/2,T_1)\cup(T_1,T/2)$, 
since $u$ symmetrically decreases away from the origin over the interval $(-T/2,T/2)$, 
\[
u(x)-u(T_1)>0 \quad \text{for $|x|<T_1$}\quad\text{and}\quad 
u(x)-u(T_1)<0 \quad \text{for $T_1<|x|<T/2$}.
\]
Hence $\phi$ cannot be orthogonal to $\{1,u\}$. 
In case $\phi$ changes signs at $x=\pm T_1, \pm T_2$, $T_1<T_2$, similarly,
$(u-u(T_1))(u-u(T_2))$ is positive for $|x|<T_1$ and $T_2<|x|<T/2$ and 
negative in $(-T_2,-T_1)\cup(T_1,T_2)$, whence $\phi$ cannot be orthogonal to $\{1, u, u^2\}$.  
A contradiction therefore leads to that the kernel of $\delta^2E$ consist merely of $u_x$.
\end{proof}

One may rerun the argument in the proof of Proposition~\ref{P:kernel} mutatis mutandis 
to obtain the nondegeneracy of the linearization associated with \eqref{E:KdV'}
at a periodic, local constrained minimizer in the range $0<\alpha\leq 2$ and $0<p<p_{max}$,
where $p_{max}$ is defined in \eqref{D:pmax}.

\medskip

Furthermore, one may verify (N2) of Assumption~\ref{A:kerL0} 
at a periodic traveling wave of \eqref{E:KdV} for $\alpha \geq -1$, at least with small amplitudes, 
whose existence follows from, e.g., a local bifurcation theorem from a simple eigenvalue, 
by explicitly calculating solution asymptotics. 

\begin{remark}[The Morse index]\label{R:index}\rm
We make a digression and characterize $n_-(\delta^2E)$ 
at a periodic, local constrained minimizer for \eqref{E:pKdV}. We begin by recalling an index formula. 

\begin{lemma}[An index formula]\label{lem:index}
Let $\mathbf{M}$ be a self-adjoint operator, bounded below and invertible with compact resolvent. 
Let $S$ be a finite-dimensional subspace of the domain of $\mathbf{M}$ and 
let $\mathbf{M}|_S$ denote the {\rm symmetric} restriction of  $\mathbf{M}$ to $S$. 
That is to say, $\mathbf{M}|_S = \Pi_S\mathbf{M} \Pi_S$, 
where $\Pi_S$ is the orthogonal projection onto $S$. Then 
\begin{equation}\label{E:index}
 n_-(\mathbf{M}) = n_-(\mathbf{M}|_S) + n_-(\mathbf{M}^{-1}|_{S^\perp}).
 \end{equation} 
\end{lemma}

Various forms of \eqref{E:index} are known in the nonlinear waves community; 
see \cite{KP, CPV, GSS}, among others. An earliest form, albeit in finite dimensions, 
is due to Haynsworth\cite{Haynsworth.1968}. 
We include a proof in Appendix~\ref{A:proof} for completeness. 

\medskip

We are going to restrict the attention to the orthogonal complement of $u_x$, 
which contains $1$ and $u$. Since $u_x$ lies in the kernel of $\delta^2E$, such restriction 
does not change $n_-(\delta^2E)$. Moreover $\delta^2E$ is invertible on $\{u_x\}^\perp$. 
Taking $S = \{1,u\}^\perp$ we apply \eqref{E:index} and write that
\[
n_-(\delta^2E) =n_-(\delta^2E|_{\{1,u\}^\perp}) + n_-((\delta^2E)^{-1}|_{\{1,u\}}).
\]
Note from \eqref{E:2negative} that $n_-(\delta^2E|_{\{1,u\}^\perp}) = 0$. 
Note moreover from \eqref{E:H3} and \eqref{E:H2} that 
$(\delta^2 E)^{-1} 1=u_a$ and $(\delta^2 E)^{-1} u=u_c$,  
where $1$ and $u$ are not in general orthogonal but, instead,  
\[
n_-\left((\delta^2E)^{-1}|_{\{1,u\}}\right) = 
n_-\left(\left(\begin{matrix} \langle 1,u_a\rangle &  \langle 1,u_c\rangle \\ 
\langle u,u_a\rangle & \langle u,u_c\rangle \end{matrix}\right)\right)= 
n_-\left(\left(\begin{matrix} M_a &  M_c \\ P_a & P_c \end{matrix}\right)\right)
\]
by Sylvester's law of inertia. Therefore
\[
n_-(\delta^2E)=
n_-\left(\left(\begin{matrix} M_a &  M_c \\ P_a & P_c \end{matrix}\right)\right).
\]
The Jacobi-Sturm sequence argument furthermore leads to that 
\begin{equation}\label{E:n-}
n_-(\delta^2E) = \# ~~\text{sign~~changes~~in } 1, M_a, M_aP_c-M_cP_a,
\end{equation}
furnishing an alternative characterization of $n_-(\delta^2 E)$. This is particularly useful in practice 
since the signs of $M_a$ and $G=M_cP_a-M_aP_c$ may be explicitly determined 
near the solitary wave limit (see Lemma \ref{L:negativity} below), near the small amplitude wave limit 
and for ODEs. 

Since $u$ may be chosen to satisfy that $u_x(0)=0$ and $u_x(0)<0$ for $0<x<T/2$ 
thanks to Proposition~\ref{P:existence}, incidentally, 
$\delta^2E$ admits at least one negative eigenvalue in $L^2_{per,even}([0,T])$. Accordingly
$\left(\begin{matrix}
         M_a & M_c\\
         P_a & P_c
         \end{matrix}\right)$
cannot be positive definite.
\end{remark}

We turn the attention to (N3) of Assumption \ref{A:kerL0}. 

\begin{lemma}[Nondegeneracy of the constraint set]\label{L:negativity}
Let $1/2<\alpha\leq 2$. 
If $u(\cdot\,;c, a, T)$ locally minimizes $H$ in $H^{\alpha/2}_{per}([0,T])$ subject to that 
$P$ and $M$ are conserved for some $c\neq 0$, $a\in \mathbb{R}$ and $T>0$
then $M_a<0$ and $G>0$ for $|a|$ sufficiently small and $T$ sufficiently large. 
\end{lemma}

\begin{proof}
Note from Proposition~\ref{P:existence} that $T>0$ is arbitrary. 
Note moreover from Galilean invariance that $a\in\mathbb{R}$ is arbitrary. 
Thanks to scaling invariance (see \eqref{E:scaling})
we may assume without loss of generality that $c=1$. Indeed \eqref{E:pKdV} remains invariant under 
\[ 
u(\cdot\,; c, a, T)\mapsto c^{-1}u(\cdot\,;1, c^{-2}a, c^{-1/\alpha}T)
\]
for any $c>0$. 

Remark~\ref{R:solitary} indicates that $u(c,a,T)$ in the range $\alpha>1/2$ 
tends to a solitary wave of \eqref{E:KdV} as $a\to 0$ and $T\to \infty$ satisfying that $aT\to 0$,
namely the solitary wave limit, which minimizes the Hamiltonian subject to constant momentum.
Consequently $P(1, a, T), P_c(1, a, T)=O(1)$ for $|a|$ sufficiently small and $T>0$ sufficiently large. 
The first identity in \eqref{E:I-solitary} moreover implies that $M(1, a, T),  M_c(1, a, T)=O(1)$ 
for $|a|$ sufficiently small, $T>0$ sufficiently large and $|aT|$ sufficiently small. 

Differentiating \eqref{E:I1-BO} with respect to $a$ and evaluating at the solitary wave limit, 
we therefore use \eqref{E:PaMc} to obtain that
\[ 
M_a=-T-2M_c=-T+O(1)<0
\]
for $T>0$ sufficiently large. Since an explicit calculation reveals that $P_c(c,a,T)>0$, moreover, 
\[ 
G=M_c^2-M_aP_c=P_cT+O(1)>0
\]
near the solitary wave limit. 
\end{proof}

In the range $\alpha<1/2$, local constrained minimizers for \eqref{E:pKdV} in the periodic wave setting
are not expected to achieve a limiting state in $H^{\alpha/2}$ as $a\to 0$ and $T\to +\infty$. 
Nevertheless, one is able to work out (N3) of Assumption~\ref{A:kerL0} at least 
at small amplitude waves, obtained via a perturbative argument, 
by explicitly calculating solution asymptotics. 

\subsection{Calculation of the modulational instability index}\label{SS:computation}
Let $u(\cdot+x_0;c,a,T)$ be a periodic traveling wave of \eqref{E:KdV}, 
satisfying \eqref{E:pKdV}, whose existence follows from, e.g., Proposition~\ref{P:existence}. 
Under Assumption~\ref{A:periodic} and Assumption~\ref{A:kerL0}
we shall take the approach in Section~\ref{S:theory} 
and determine its spectral instability near the origin to long wavelengths perturbations. 
In particular we shall calculate the modulational instability index $\Delta$, defined in \eqref{D:index}, 
in terms of $U$, $P$, $M$ as functions of $c$ and~$a$. 
(Note that $x_0\in\mathbb{R}$ and $T>0$ are arbitrary.) 
Incidentally $U$, $P$, $M$ correspond, respectively, to the third, the second, the first momenta. 
We may express the result in terms of $H$, $P$, $M$, instead,
noting from \eqref{E:HKdV} and \eqref{E:I3-BO} that
\[
K=3H-2cP-aM \quad\text{and}\quad U=-2H+2cP+aM.
\] 
Later in Section~\ref{S:gKdV} we shall calculate the index, 
in the case of general nonlinearities, in terms of $K$ and $U$, $P$, $M$ 
together with their derivatives with respect to $c$, $a$ and $T^{-1}$. 

\medskip

Recall that (see \eqref{E:L+KdV})
\begin{align}
L_0(u;c,a)=&J\delta^2E(u;c,a)=\partial_x(\Lambda^\alpha-2u-c), \notag 
\intertext{and we make an explicit calculation to find that}
L_1(u;c,a)=&[L_0,x]=(\alpha+1)\Lambda^\alpha-2u-c, \label{E:L1} \\
L_2(u;c,a)=&[L_1,x]=\alpha(\alpha+1)\Lambda^{\alpha-2}\partial_x.\label{E:L2}
\end{align}
Recall moreover that (see Lemma~\ref{L:L0})
\begin{subequations}
\begin{alignat}{2}
v_1 &= u_a, \qquad &&w_1=M_cu-P_c, \label{E:vw1KdV}\\
v_2 &= u_x, \qquad &&w_2= \partial_x^{-1} (M_a u_c - M_cu_a),\label{E:vw2KdV}\\
v_3 &= u_c, \qquad &&w_3=P_a-M_au \label{E:vw3KdV}
\end{alignat}\end{subequations}
satisfy \eqref{E:vw1}-\eqref{E:vw3} and 
\[
\l v_j, w_k\r=\{M,P\}_{c,a}\delta_{jk}=G\delta_{jk},\qquad j,k=1,2,3,
\]
where 
\[ 
\{f,g\}_{x,y}=f_xg_y-f_yg_x,
\]
$\delta_{jk}=1$ if $j=k$ and $\delta_{jk}=0$ otherwise. 
Furthermore they satisfy \eqref{E:parity}. 

\medskip

We begin by rewriting $\l w_j,L_1v_k\r$, $j,k=1,2,3$, 
in terms of $U,\,P,\,M$ as functions of $c$ and $a$. 
We may use \eqref{E:pKdV} to write (see \eqref{E:L1}) that
\begin{align}
L_1u=(\alpha+1)\Lambda^\alpha u-2u^2-cu
=&(\alpha+1)\delta K+2\delta U-c\delta P \notag \\
=&(1-\alpha)\delta U+\alpha c\delta P+(\alpha+1)a \delta M. \label{E:L1u}
\end{align}
Moreover
\begin{equation}\label{E:L11}
L_11=(\alpha+1)\Lambda^\alpha 1-2u-c=-(2\delta P+c\delta M).
\end{equation}
Since $L_1$ is self-adjoint, 
we use \eqref{E:vw1KdV} and \eqref{E:L1u}, \eqref{E:L11} to calculate that
\begin{subequations}
\begin{align}
\l w_1, L_1v_1\r=&\l L_1w_1, v_1\r \notag \\
=&M_c\big((1-\alpha)U_a+\alpha cP_a+(\alpha+1)aM_a\big)+P_c\big(2P_a+cM_a\big). \label{E:11}
\end{align}
Similarly
\begin{align}
\l w_1, L_1v_3\r=&\,\quad M_c\big((1-\alpha)U_c+\alpha cP_c+(\alpha+1)aM_c\big)\,+\,P_c\big(2P_c+cM_c\big),\label{E:13} \\ 
\l w_3, L_1v_1\r=&-M_a\big((1-\alpha)U_a+\alpha cP_a+(\alpha+1)aM_a\big)-P_a\big(2P_a+cM_a\big), \label{E:31}\\ 
\l w_3, L_1v_3\r=&-M_a\big((1-\alpha)U_c+\alpha cP_c+(\alpha+1)aM_c\big)\,-\,P_a\big(2P_c+cM_c\big). \label{E:33}
\end{align}
\end{subequations}
Adding \eqref{E:11} and \eqref{E:33},
\begin{equation}\label{E:11+33} 
\l w_1, L_1v_1\r +\l w_3, L_1v_3\r=(1-\alpha)\big(\{M, U\}_{c,a}-cG\big).\end{equation}
Moreover we use \eqref{E:vw2KdV} and make an explicit calculation to obtain that
\begin{align}\label{E:22}
\l w_2, L_1v_2\r=&\l \partial_x^{-1}(M_au_c-M_cu_a), L_1u_x\r \notag \\
=&\l  \partial_x^{-1}(M_au_c-M_cu_a), \partial_x((\alpha+1)\Lambda^{\alpha}-u-c)u\r \notag\\
=&-\l M_au_c-M_cu_a, -\alpha \delta U+\alpha c\delta P+(\alpha+1)a\delta M\r \notag\\
=&-\alpha\big(\{M, U\}_{c,a}-cG\big). 
\end{align}
Adding \eqref{E:11+33} and \eqref{E:22}, therefore (see \eqref{D:B})
\begin{equation}\label{E:B-BO}
D_2=\boxed{G^2(1-2\alpha)\big(\{M,U\}_{c,a}-cG\big)}.
\end{equation}

\medskip

Calculations of $c_{22}$ and $c_{32}$ in \eqref{D:c22} and \eqref{D:c32} are involved. 
Below we combine a Pohozaev type identity in Lemma~\ref{L:pohozaev1}
and scaling invariance to rewrite $L_0^{-1}L_1v_2$ in a convenient form.

\begin{lemma}\label{L:pohozaev'}Under Assumption~\ref{A:kerL0}, 
\begin{subequations}
\begin{align}
L_0^{-1}L_1v_2=
&-\alpha u+G^{-1}\l w_2,L_1v_2\r v_3+\alpha G^{-1}\big(\l w_1,u\r v_1+\l w_3, u\r v_3\big) \label{E:L0L1v2}\\
=&-\alpha u+\alpha G^{-1}(2M_cP-P_cM)v_1-\alpha cv_3.\label{E:L0L1v2'} 
\end{align}
\end{subequations}
\end{lemma}

\begin{proof}
Thanks to scaling invariance (see \eqref{E:scaling}), if $u(x; c, a, T)$ satisfies \eqref{E:pKdV} 
then so does $\lambda^\alpha u(\lambda x; \lambda^\alpha c, \lambda^{2\alpha}a, \lambda T)$ 
for any $\lambda>0$. Differentiating  
\[ 
\delta E(\lambda^\alpha u(\lambda x; \lambda^\alpha c, \lambda^{2\alpha}a, \lambda T))=0
\]
with respect to $\lambda$ and evaluating at $\lambda=1$, therefore, we find that 
\[ 
\delta^2E(\alpha u+xu_x+\alpha cu_c+2\alpha au_a+Tu_T)=0.
\]
In other words, $\alpha u+2\alpha au_a+\alpha cu_c+xu_x+Tu_T$ lies in the kernel of $\delta^2E$. 
On the other hand, (N2) of Assumption~\ref{A:kerL0} dictates that 
$\ker(\delta^2E)$ is one-dimensional and spanned by $u_x$, which is odd. 
Since $\alpha u+2\alpha au_a+\alpha cu_c+xu_x+Tu_T$ is even, though,
it must be zero. Consequently $xu_x+Tu_T=-\alpha(u+2au_a+cu_c)$. 

To proceed, we infer from \eqref{E:pohozaev} that 
\[
L_1v_2=L_0\big(-\alpha(u+2av_1+cv_3)\big),
\]
whence
\[
L_0^{-1}L_1v_2=-\alpha u+c_1v_1+c_3v_3
\]
for some $c_1 \in \mathbb{C}$ a constant and $c_3=-\alpha c$. Indeed, $\ker(L_0)=\spn\{v_1, v_2\}$
while $v_2$ is odd. Taking inner products against $w_1$ and $w_3$ then leads to that 
\begin{equation}\label{E:c1c3}
0=-\alpha \l w_1,u\r+c_1G \quad\text{and}\quad \l w_2, L_1v_2\r=-\alpha\l w_3, u\r+c_3G.
\end{equation}
Therefore \eqref{E:L0L1v2} follows. Moreover \eqref{E:L0L1v2'} follows, since 
\begin{equation}\label{E:w13u}
\l w_1, u\r=2M_cP-P_cM\quad\text{and}\quad \l w_3,u\r=P_aM-2M_aP.
\end{equation}
\end{proof}

Introducing (see \eqref{E:L1} and \eqref{E:L2}) 
\begin{align}\label{E:deltaV}
\delta W
:=&\alpha L_1u+\frac12L_2u_x=\alpha \Big(\frac12(\alpha+1)\Lambda^{\alpha}u-2u^2-cu\Big)\notag \\
=&\frac12\alpha\big((3-\alpha)\delta U+(\alpha-1)c\delta P+(\alpha+1)a\delta M\big)
\end{align}
we use \eqref{E:L0L1v2} to revamp $c_{22}$, in \eqref{D:c22}, as
\begin{align}\label{D:c22'}
c_{22}=&\l w_1, \alpha L_1u\r-G^{-1}\l w_1, L_1\l w_2, L_1v_2\r v_3\r \notag \\
&-\alpha G^{-1}\l w_1, L_1(\l w_1,u\r v_1+\l w_3, u\r v_3)\r+\frac12\l w_1, L_2v_2\r \notag \\
=:&-G^{-1}\l w_2, L_1v_2\r\l w_1, L_1 v_3\r +d_{22},
\intertext{where, substituting \eqref{E:deltaV}, \eqref{E:w13u}, \eqref{E:11}, \eqref{E:13} and \eqref{E:I1-BO},} 
d_{22}=&\l w_1,\delta W\r
-\alpha G^{-1}\big(\l w_1,u\r\l w_1, L_1v_1\r+\l w_3, u\r\l w_1, L_1v_3\r\big)\label{D:d22} \\
=&\frac12\alpha\big(M_c(3(3-\alpha)U+2(\alpha-1)cP+(\alpha+1)aM)+2P_c(4P+cM)\big)
\hspace*{-.1in}\notag \\
&-\alpha(1-\alpha)G^{-1}\big(2M_cP\{M,U\}_{c,a}-M_cM\{P,U\}_{c,a}\big) \notag \\
&-\alpha\big(2\alpha cM_cP+4P_cP+(\alpha+1)\alpha aM_cM+cP_cM\big).\notag
\end{align}
Similarly (see \eqref{D:c32})
\begin{align}
c_{32}=:&-G^{-1}\l w_2, L_1v_2\r\l w_3, L_1 v_3\r +d_{32}, \label{D:c32'}
\intertext{where, substituting \eqref{E:deltaV}, \eqref{E:w13u}, \eqref{E:31}, \eqref{E:33} and \eqref{E:I1-BO},}
d_{32}=&\l w_3, \delta W\r
-\alpha G^{-1}\big(\l w_1,u\r\l w_3, L_1v_1\r+\l w_3, u\r \l w_3, L_1v_3\r\big)\label{D:d32} \\
=&\frac12\alpha\big(-2P_a(4P+cM)-M_a(3(3-\alpha)U+2(\alpha-1)cP+(\alpha+1)aM)\big)
\hspace*{-.1in} \notag  \\
&+\alpha(1-\alpha)G^{-1}\big(2M_aP\{M,U\}_{c,a}-M_aM\{P,U\}_{c,a}\big) \notag \\
&+\alpha\big(2\alpha cM_aP+4P_aP+(\alpha+1)\alpha a M_aM+cP_aM\big).\notag
\end{align}

Since \eqref{E:11} through \eqref{E:33} imply that 
\begin{align}\label{E:Gamma}
\l w_1, L_1v_3\r&\l w_3, L_1v_1\r-\l w_1, L_1v_1\r\l w_3, L_1v_3\r\notag \\
=(1-&\alpha)G\big(2\{P,U\}_{c,a}+c\{M,U\}_{c,a}\big)-G^2\big(2(\alpha+1)a-\alpha c^2\big)=:\Gamma,
\end{align}
we deduce from \eqref{D:C}, \eqref{D:D} and \eqref{D:c22'}, \eqref{D:c32'} that
\begin{align}\label{E:C}
D_1=&G\big(\l w_1,L_1v_3\r\l w_3, L_1v_1\r-\l w_1, L_1v_1\r\l w_3, L_1v_3\r \notag \\ 
&\quad-\l w_2, L_1v_2\r\left(\l w_1, L_1v_1\r+\l w_3, L_1v_3\r\right)-Gd_{32}\big) \notag \\
=&\boxed{G\big(\Gamma+\alpha(1-\alpha)(\{M,U\}_{c,a}-cG)^2-Gd_{32}\big)}
\end{align}
and 
\begin{align}\label{E:D}
D_0=&\l w_2, L_1v_2\r\big(\l w_1, L_1v_1\r\l w_3, L_1v_3\r-\l w_1, L_1v_3\r \l w_3, L_1v_1\r\big)\notag \\ 
&+G\big(\l w_3, L_1v_1\r d_{22}-\l w_1, L_1v_1\r d_{32}\big) \notag \\
=&\boxed{\alpha\Gamma(\{M,U\}_{c,a}-cG)+G\big(\l w_3, L_1v_1\r d_{22}-\l w_1, L_1v_1\r d_{32}\big)},
\end{align}
where $\l w_1, L_1v_1\r$, $\l w_3, L_1v_1\r$ and $d_{22}$, $d_{32}$ are specified, respectively, 
by calculating the right sides of \eqref{E:11}, \eqref{E:31} and \eqref{D:d22}, \eqref{D:d32}, 
in terms of $U$, $P$, $M$ as functions of $c$ and $a$. 

\medskip

To summarize, the modulational instability index, defined in \eqref{D:index}, 
at a periodic traveling wave of \eqref{E:KdV} is obtained as the discriminant of the cubic polynomial 
\[ 
\det({\bf D}-\mu G{\bf I})=-G^3\mu^3+D_2\mu^2+D_1\mu+D_0,
\]
where $D_2$, $D_1$, $D_0$ are specified in \eqref{E:B-BO}, \eqref{E:C}, \eqref{E:D}, respectively,
in terms of $U$, $P$, $M$ as functions of $c$ and $a$.

Incidentally the effective dispersion matrix $\mathbf{D}$, defined in \eqref{D:dispersion},
may be expressed with the help of \eqref{D:c22'} and \eqref{D:c32'} as 
\begin{equation}\label{D:dispersion'}
\mathbf{D}= \left( \begin{matrix}
\l w_3, L_1v_3\r & Gd_{32} & \l w_3, L_1v_1 \r  \\
1 & \l w_2, L_1v_2\r & 0 \\
\l w_1, L_1v_3\r & Gd_{22} & \l w_1, L_1v_1\r \end{matrix}\right).
\end{equation}

\subsection{Evaluation at Benjamin-Ono, periodic traveling waves}\label{SS:BO}
The formulae in the previous subsection may simplify with the help of analytical solutions, 
which we shall illustrate by discussing in the case of $\alpha=1$, namely the Benjamin-Ono equation.
In particular we shall evaluate the effective dispersion matrix $\mathbf{D}$, 
defined in \eqref{D:dispersion'}, as a function of $c$ and $a$, 
at a periodic traveling wave (see \eqref{E:pBO} or \cite{Benjamin}, for instance)
\begin{equation}\label{E:pBOkappa}
u(x;c,a, T)=\frac{\displaystyle \frac{\kappa^2}{\sqrt{c^2-4a-\kappa^2}}}
{\sqrt{\displaystyle \frac{c^2-4a}{c^2-4a-\kappa^2}}-\cos(\kappa x)}
-\frac12(\sqrt{c^2-4a}+c).
\end{equation}
Here $\kappa=2\pi/T$, where $T>0$ is arbitrary but fixed, and
\begin{equation}\label{E:c-range}
c<0 \quad\text{and}\quad c^2-4a-\kappa^2>0.
\end{equation}
Since the Benjamin-Ono equation obeys Galilean invariance under
\[
 u(x;c-2s,a-cs+s^2) = u(x;c,a)+s, 
\]
upon an appropriate choice of $s\in\mathbb{R}$, one may assume that $a=0$ and $c<0$.

\medskip

Since 
\begin{align*}
\int^{2\pi}_0 \frac{dx}{b-\cos(x)}\quad=&\frac{2\pi}{\sqrt{b^2-1}}, \\
\int^{2\pi}_0 \frac{dx}{(b-\cos(x))^2}=&\frac{2\pi b}{(b^2-1)^{3/2}},\\
\int^{2\pi}_0 \frac{dx}{(b-\cos(x))^3}=&\frac{\pi(2b^2+1)}{(b^2-1)^{5/2}}
\end{align*}
for $b>1$ a constant, we calculate that
\begin{alignat*}{2}
M(c,a,T):=&\int^T_0u(x;c,a)~dx&&=2\pi-\frac{T}{2}(\sqrt{c^2-4a}+c), \\
P(c,a,T):=&\int^T_0\frac12u^2(x;c,a)~dx&&=-\pi c+\frac{T}{8}(\sqrt{c^2-4a}+c)^2, \\
U(c,a,T):=&\int^T_0-\frac13u^3(x;c,a)~dx &&=\frac{\pi \kappa^2}{3}-\pi(c^2-2a)+\frac{T}{24}(\sqrt{c^2-4a}+c)^3.
\end{alignat*}
They reduce at $a=0$ to 
\begin{equation}\label{E:MPU}
M(c,0,T)=2\pi,\qquad P(c,0,T)=-\pi c, \qquad U(c,0,T)=\frac{\pi\kappa^2}{3}-\pi c^2.
\end{equation}
Differentiating $M(c,a,T)$ and $P(c,a,T)$ with respect to $c$ and $a$, moreover, 
\begin{alignat*}{2}
M_c(c,a,T)&=-\frac{T}{2}\frac{\sqrt{c^2-4a}+c}{\sqrt{c^2-4a}}, \qquad &
M_a(c,a,T)&=\frac{T}{\sqrt{c^2-4a}},  \\
P_c(c,a,T)&=-\pi+\frac{T}{4}\frac{(\sqrt{c^2-4a}+c)^2}{\sqrt{c^2-4a}}, \quad &
P_a(c,a,T)&=-\frac{T}{2}\frac{\sqrt{c^2-4a}+c}{\sqrt{c^2-4a}},
\end{alignat*}
which reduce at $a=0$ to 
\begin{equation}\label{E:MP_ca}
M_c(c,0,T)=P_a(c,0,T)=0,\qquad M_a(c,0,T)=-\frac{T}{c}, \qquad P_c(c,0,T)=-\pi.
\end{equation}
Differentiating $U(c,a,T)$ with respect to $c$ and $a$, similarly, 
\[
U_c(c,a,T)=-2\pi c+\frac{T}{8}\frac{(\sqrt{c^2-4a}+c)^3}{\sqrt{c^2-4a}}
\quad\text{and}\quad U_a(c,a,T)=2\pi-\frac{T}{4}\frac{(\sqrt{c^2-4a}+c)^2}{\sqrt{c^2-4a}},
\]
which reduce at $a=0$ to 
\begin{equation}\label{E:U_ca}
U_c(c,0,T)=-2\pi c\quad\text{and}\quad U_a(c,0,T)=2\pi.
\end{equation}
Consequently
\begin{equation}\label{E:G-BO}
G(c,a,T)=(M_cP_a-M_aP_c)(c,a)=\frac{\pi T}{\sqrt{c^2-4a}}>0,
\end{equation}
$G(c,0,T)=-\pi T/c$ and $\{M,U\}_{c,a}(c,0,T):=(M_cU_a-M_aU_c)(c,0,T)=-2\pi T$. 

\medskip

To proceed, we substitute \eqref{E:MP_ca}, \eqref{E:U_ca}, \eqref{E:G-BO} and $\alpha=1$ 
into \eqref{E:11}-\eqref{E:33}, \eqref{E:22}, respectively, to calculate that
\begin{alignat*}{2}
\langle w_1, L_1v_1\rangle(c,0,T)&=\pi T, \qquad &
\langle w_1, L_1v_3\rangle(c,0,T)&=2\pi^2, \\
\langle w_3, L_1v_1\rangle(c,0,T)&=0, \qquad &
\langle w_3, L_1v_3\rangle(c,0,T)&=-\pi T, \\
\langle w_2, L_1v_2\rangle(c,0,T)&=\pi T. & &
\end{alignat*}
Moreover we substitute \eqref{E:MPU}, \eqref{E:MP_ca}, \eqref{E:U_ca}, \eqref{E:G-BO} 
and $\alpha=1$ into \eqref{D:d22} and \eqref{D:d32}, respectively, to calculate that 
\[
d_{22}(c,0,T)=0\quad\text{and}\quad d_{32}(c,0,T)=\frac{\pi T}{c}(\kappa^2-c^2).
\]

\medskip

To summarize, the effective dispersion matrix (see \eqref{D:dispersion'})  
at the Benjamin-Ono, periodic traveling wave $u(\cdot\,;c,0,T)$, in \eqref{E:pBOkappa}, 
is explicitly calculated as
\[
\mathbf{D}(u; c,0,T)=\left(\begin{matrix} 
-\pi T & (\pi T)^2(1-(\frac{\kappa}{c})^2) & 0 \\
1 & \pi T & 0 \\ 2\pi^2 & 0 & \pi T \end{matrix}\right),
\]
whose eigenvalues are $\pi T$ and $\pm\pi T\sqrt{2-(\frac{\kappa}{c})^2}$. 
Since the underlying wave exists in the range $c^2>\kappa^2$ (see \eqref{E:c-range}), 
moreover, the expression under the radical is positive. 
In light of Corollary~\ref{T:instability}, therefore,
a periodic traveling wave of the Benjamin-Ono equation is modulationally stable.

\section{Numerical Experiments}\label{S:numerical}
In many examples of interest, analytical expressions for periodic traveling waves 
of Hamiltonian systems are not available in closed form. Hence one does not expect to simplify
formulae in Section~\ref{SS:perturbation} and Section~\ref{SS:computation}. 
The present development is well-suited to numerical calculations, nevertheless, 
since formulae may be expressed in terms of conserved quantities of the underlying,
periodic traveling wave and their derivatives with respect to Lagrange multipliers, 
which are easily approximated by computational methods. Here we conduct 
preliminary numerical experiments of modulational instability in the family \eqref{E:KdV}. 

\medskip
 
The Petviashvili iteration (see \cite{petviashvili.1976}, for instance) is a commonly used, 
numerical method of generating solitary waves of nonlinear dispersive equations.  
We shall modify the method to numerically generate periodic solutions of \eqref{E:pKdV}.
An obvious strategy is to iterate 
\begin{equation}\label{PetIteration}
\Lambda^\alpha u_{n+1} =u_n^2 + c u_n + a, \qquad n=0,1,2,\dots,
\end{equation}
i.e., the standard Petviashvili iteration. Unfortunately it is complicated by that 
the kernel of $\Lambda^\alpha$ is non-trivial. 
For one thing, we must impose a solvability condition for $u_{n+1}$. 
Another, related, is that $u_{n+1}$ is defined merely up to an element in the kernel. 
In order to address these issues,  we choose the projection onto $\ker(\Lambda^\alpha)$ 
of the $n$-th iterate $u_n$ so that $u_{n}^2 + cu_{n} + a$ is orthogonal to the kernel 
and therefore we may solve \eqref{PetIteration} for $u_{n+1}$.
Specifically, if $v_{n+1}$ is the $\ker(\Lambda^\alpha)$-orthogonal component 
of $\Lambda^{-\alpha}(u_n^2+cu_n+a)$ and if $\phi$ is a unit vector in the kernel then 
\begin{equation}\label{eqn:pet}
u_{n+1}=v_{n+1} + \theta_{n+1} \phi,\quad\text{where }
\theta_{n+1} =  \frac12\big(-c \pm \sqrt{c^2 - 4(a +\|v_{n+1}\|_{L^2}^2)} \big),
\end{equation}
guarantees that $u_{n+1}^2 + c u_{n+1}+a$ is orthogonal to  $\ker(\Lambda^\alpha).$
Note that \eqref{eqn:pet} yields two solutions, different in the direction spanned by the kernel. 
But the iteration converges for at most one solution, though. 

We implemented \eqref{eqn:pet} in Mathematica spectrally 
using the discrete cosine Fourier transform (of type I). Specifically we solved 
\[
u_{n+1} = (\Lambda^\alpha)_{MP}^{-1}(u_n^2 + cu_n + a)+\theta_{n+1},
\]
where $(\Lambda^\alpha)_{MP}^{-1}$ denotes the Moore-Penrose psuedo-inverse,
defined via the Fourier series as 
\[
(\Lambda^\alpha)_{MP}^{-1}\cos(kx)=k^{-\alpha} \cos(kx), \quad k\neq 0 
\quad\text{and}\quad (\Lambda^\alpha)_{MP}^{-1}(1)=0.
\] 
Note that the cosine Fourier transform enforces evenness and thus 
breaks invariance under spatial translations. 
We set $v_0(x) = \cos(2\pi x/T)$, where $T>0$ is the period, 
and we continued the iteration until $\|u_{n+1}-u_n\|_{L^2}=O(10^{-14})$. 
In practice the algorithm appeared to converge, although convergence was slow at times, 
a well-known drawback of the Petviashvili iteration;
see \cite{lakoba.yang.2007,lakoba.yang.2008}, for instance, 
for a discussion of the convergence rate of the method.

\medskip

In the first set of numerical experiments we benchmark the method by attempting to reproduce 
known solutions of the KdV equation, namely cnoidal waves, satisfying that
\begin{equation}\label{eqn:KdV}
-u_{xx}=u^2+cu+a, \qquad u(x+T) =u(x).
\end{equation}
We record two solutions in terms of Jacobi elliptic function:
\begin{itemize}
 \item Experiment 1ab: $c=-2$, $a=-2$, $T=2K(\sqrt{2}/2)\approx 3.708$,  
 for which an exact solution is $u(x) = 1+3{\rm cn}^2(x + K(\sqrt{2}/2),\sqrt{2}/2)$;
 \item Experiment 2ab: $c=\frac83$, $a=-5/3$, $T=2K(\sqrt{6}/6)\approx 3.296$, 
 for which an exact solution is $u(x) = {\rm cn}^2(x + K(\sqrt{6}/6),\sqrt{6}/6)$.
\end{itemize}
Here $K$ represents the complete elliptic integral of the first kind.
Note that solutions of the boundary value problem \eqref{eqn:KdV} are not unique. 
Elliptic functions are specified in terms of the elliptic modulus $k$. 
Note however that Mathematica works with the parameter $k^2$. 

The results from numerical experiments are in Figure~1. 
The graphs on the left represent profiles of numerically generated, traveling waves 
while those on the right represent the difference between the numerically generated solution 
and the analytical solution, computed using Mathematica's built-in elliptic function routines. 
The agreement is excellent, and the method appears to converge to the 
appropriate analytical solutions, with pointwise errors of the order of $10^{-14}$ or less.
\begin{figure}[ht]\label{fig:exp1}
\begin{center}
\includegraphics[width=2in]{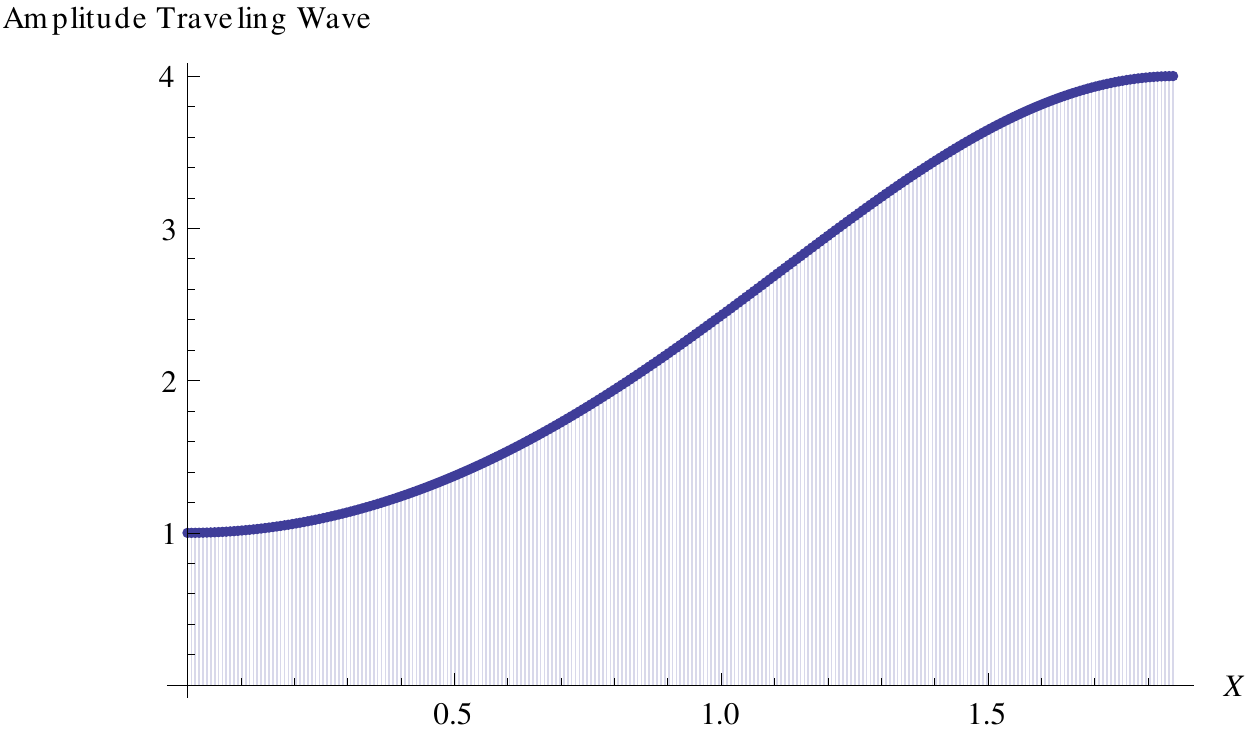} \includegraphics[width=2in]{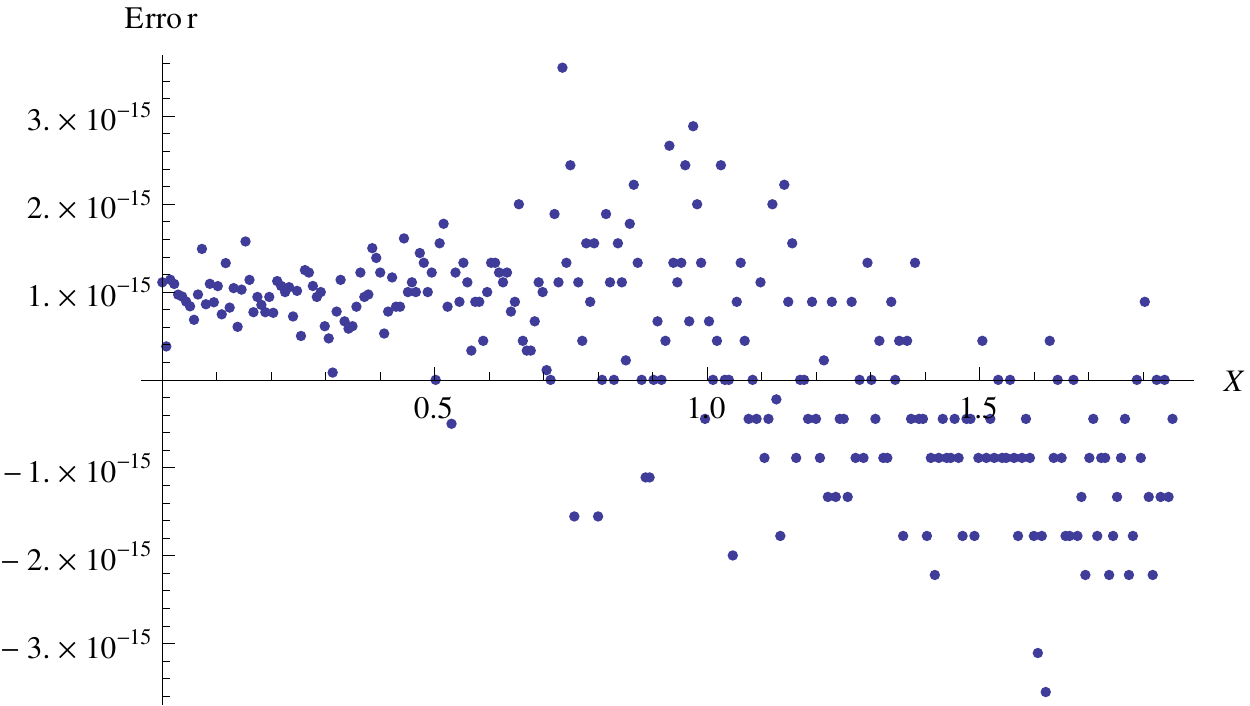} \\
\includegraphics[width=2in]{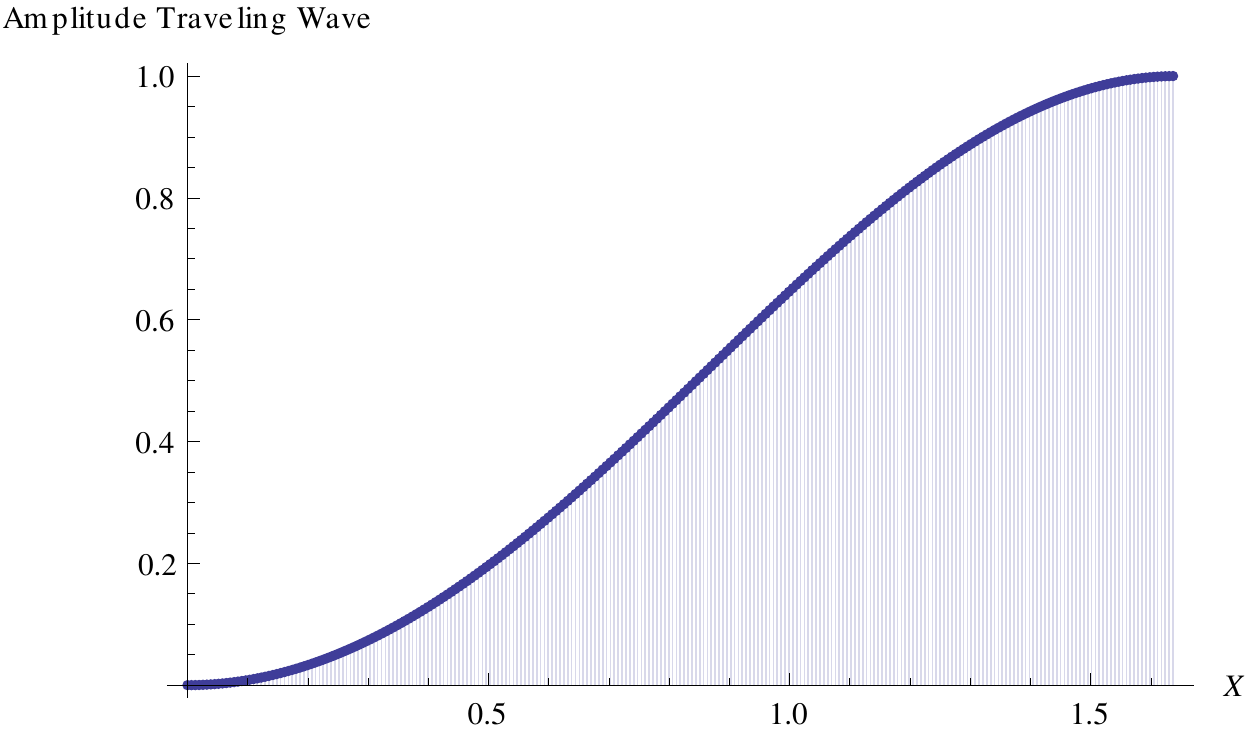} \includegraphics[width=2in]{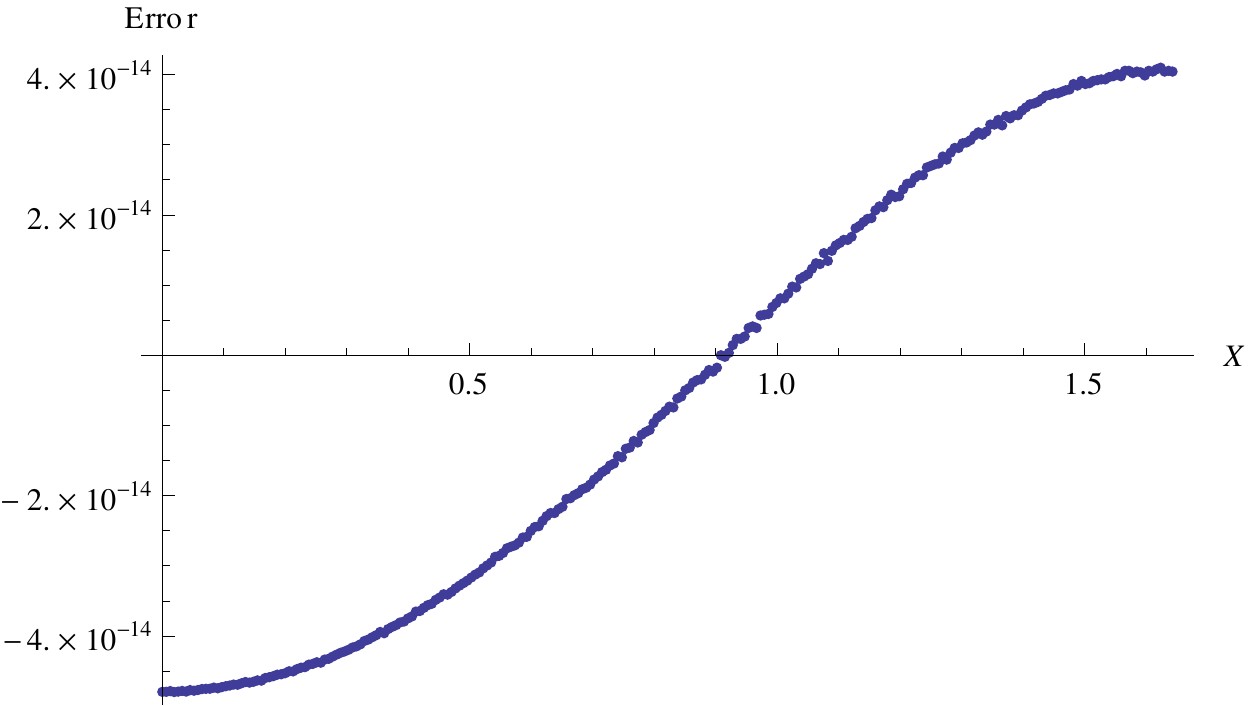} 
\caption{The graphs on the left represent numerically generated, periodic traveling waves 
of the KdV equation. The graphs on the right represent the difference between the 
numerical and analytical solutions. The solutions are depicted over half the period.}
\end{center}
\end{figure}
\noindent 

\medskip

In the second experiment 
we attempt to reproduce periodic traveling waves of the Benjamin-Ono equation. 
We choose $c=-5$, $a=0$, $T=\pi/2$ and 
we recall from \eqref{E:pBO} or \cite{Benjamin}, for instance, that 
\[
 u(x;5,0,\pi/2) = \frac{16}{5 + 3 \cos(4 x)} 
\]
solves 
\[
H u_x = u^2 + 5 u, \qquad u(x+\pi/2)=u(x).
\]
The results from the numerical experiment are in Figure~2. 
The agreement between the numerical and analytical solutions is excellent. 
\begin{figure}[ht]\label{fig:exp2}
\begin{center}
\includegraphics[width=2in]{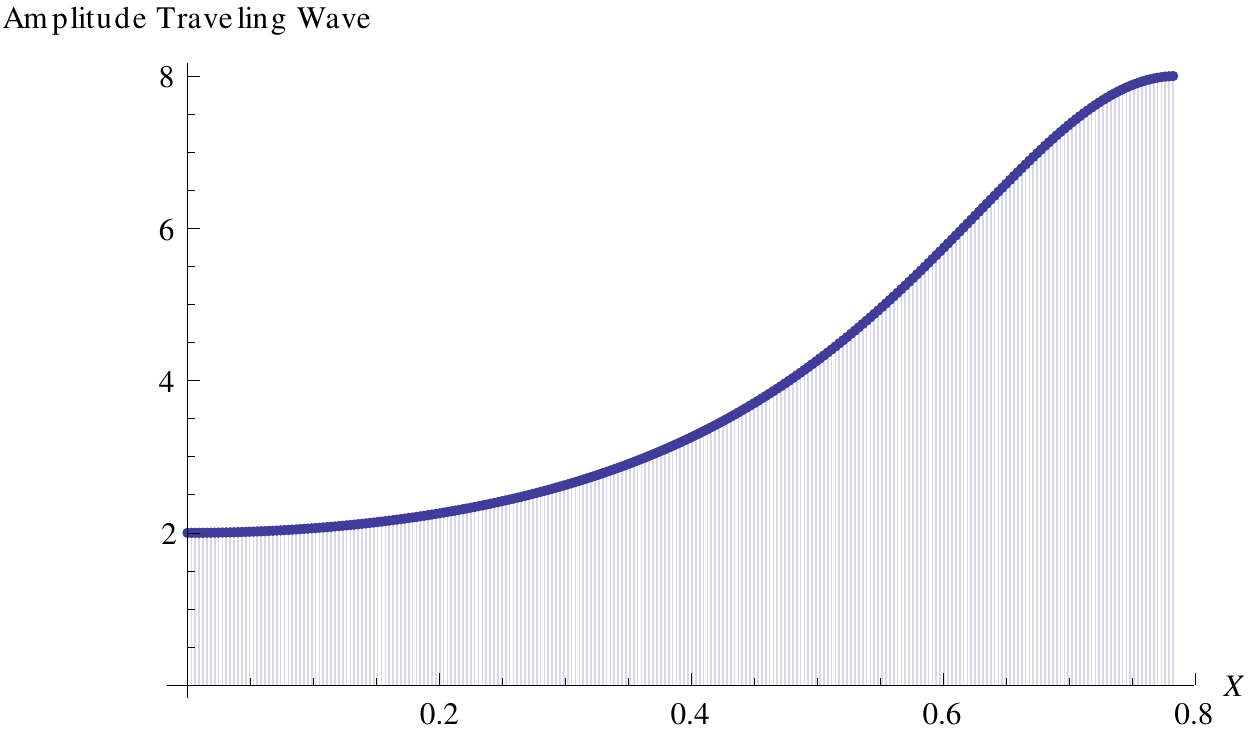} \includegraphics[width=2in]{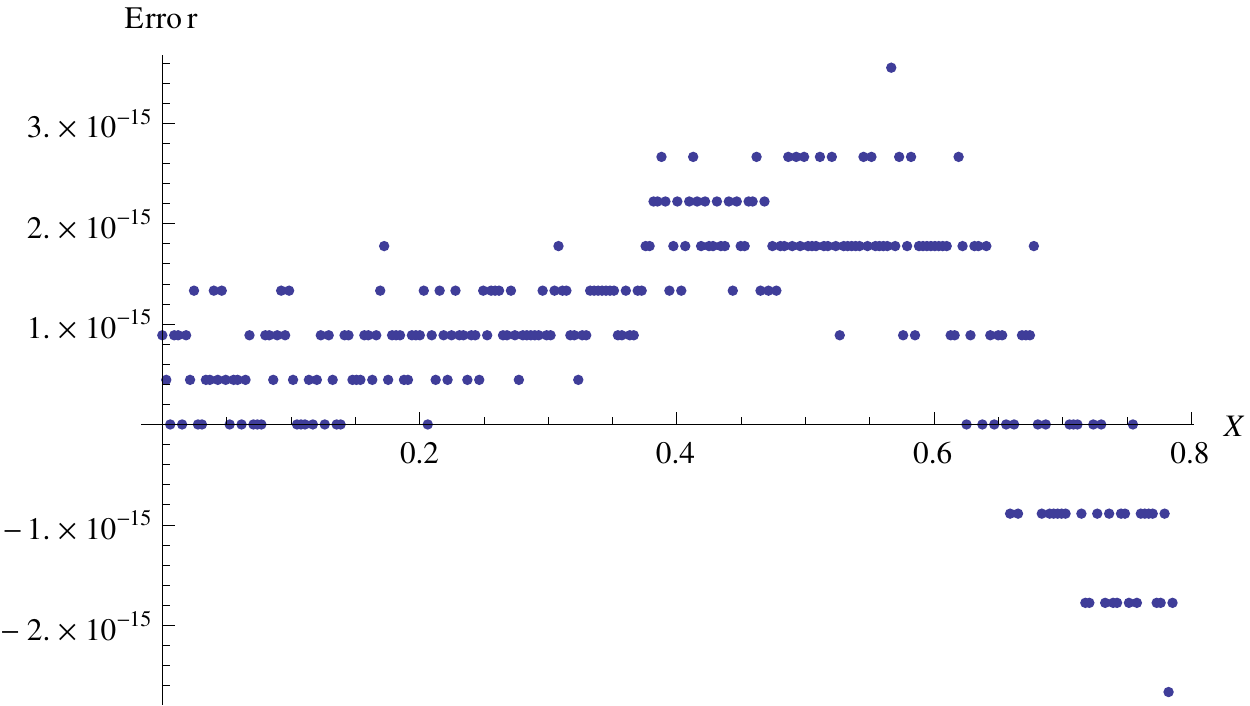} \\
\caption{The graph on the left represents the numerically generated, periodic traveling wave of 
the Benjamin-Ono equation. The graph on the right represents the difference 
between the numerical and analytical solutions.}
\end{center}
\end{figure}

We promptly use the numerical routines to explore spectral instability  for \eqref{E:KdV} 
near the origin to long wavelengths perturbations.
Specifically we shall numerically implement formulae in Section~\ref{SS:computation}. 
We compute $M$, $P$, $U$ from numerically generated solutions of \eqref{E:pKdV}
using a trapezoidal rule and numerically differentiate them using a two point symmetric stencil. 

In the case of the KdV equation 
\[
u_{xx} =u^2-2 u -2, \qquad u(x+2K(\sqrt{2}/2))=u(x),
\]
for instance, the effective dispersion matrix in \eqref{D:dispersion'} is numerically found to be 
\[
{\bf D} = \left(
\begin{array}{ccc}
 -17.0603 & -70.4702 & -4.82493 \\
 1 & 24.4707 & 0 \\
 33.8253 & -446.679 & 29.2956
\end{array}
\right),
\]
whose eigenvalues are  $\mu \approx 31.64$, $\mu \approx 14.08$, $\mu \approx -9.016$. 
Therefore the underlying, periodic traveling wave is modulationally stable,
reproducing the known result. 

In the case of the Benjamin-Ono equation 
\[
H u_x = u^2 - 5u, \qquad u(x+\pi/2)=u(x),
\]
for instance, the effective dispersion matrix is approximately
\[
{\bf D}=  \left(
\begin{array}{ccc}
 -4.9348 & 8.76682 & \text{2.7902947983404674$\grave{ }$*${}^{\wedge}$-11} \\
 1 & 4.9348 & 0 \\
 19.7392 & \text{7.435829729729448$\grave{ }$*${}^{\wedge}$-10} & 4.9348
\end{array}
\right),
\]
whose eigenvalues are $\mu \approx 5.755$, $\mu \approx -5.755$, $\mu \approx 4.935$. 
Therefore the underlying, periodic traveling wave is modulationally stable. 
This is in excellent agreement with the results of Section~\ref{SS:BO}, 
where $\mu_1 = \frac{\pi^2\sqrt{34}}{10}\approx 5.755$, $\mu_2 = -\frac{\pi^2\sqrt{34}}{10}\approx -5.755$, $\mu_3 = \frac{\pi^2}{2} \approx 4.935$.

\medskip

We examined parameter values in the range $1<\alpha<2$ and found, interestingly, 
a modulationally unstable wave. We chose $c=-5$, $a=0$, $T=\pi/2$ 
and allowed $\alpha$ to vary between $\alpha\approx.95$ and $\alpha\approx 1.17$. 
Following the solution branch, below $\alpha\approx.95$ the numerical scheme failed to converge 
whereas above $\alpha \approx 1.17$ it converged to the constant solution $u=5$. 
We found that the eigenvalues depended rather sensitively upon the parameter $\alpha$, 
and the underlying, periodic traveling wave became unstable at $\alpha \approx 1.025$. 

The results from the numerical experiment are in Figure~3. 
The graph on the left side of the figure represents 
the imaginary part of the eigenvalue as a function of $\alpha$. 
It is zero for $\alpha \lessapprox 1.025$, implying stability, and increases beyond the exponent.
The graph on the right represents the profile of the periodic traveling wave at the onset of instability, $\alpha \approx 1.025$. It seems shallower than the corresponding, periodic traveling wave of the Benjamin-Ono equation. But it seems to arise, though. 
\begin{figure}[ht]\label{UnstableFig}
\begin{center}
\includegraphics[width=2in]{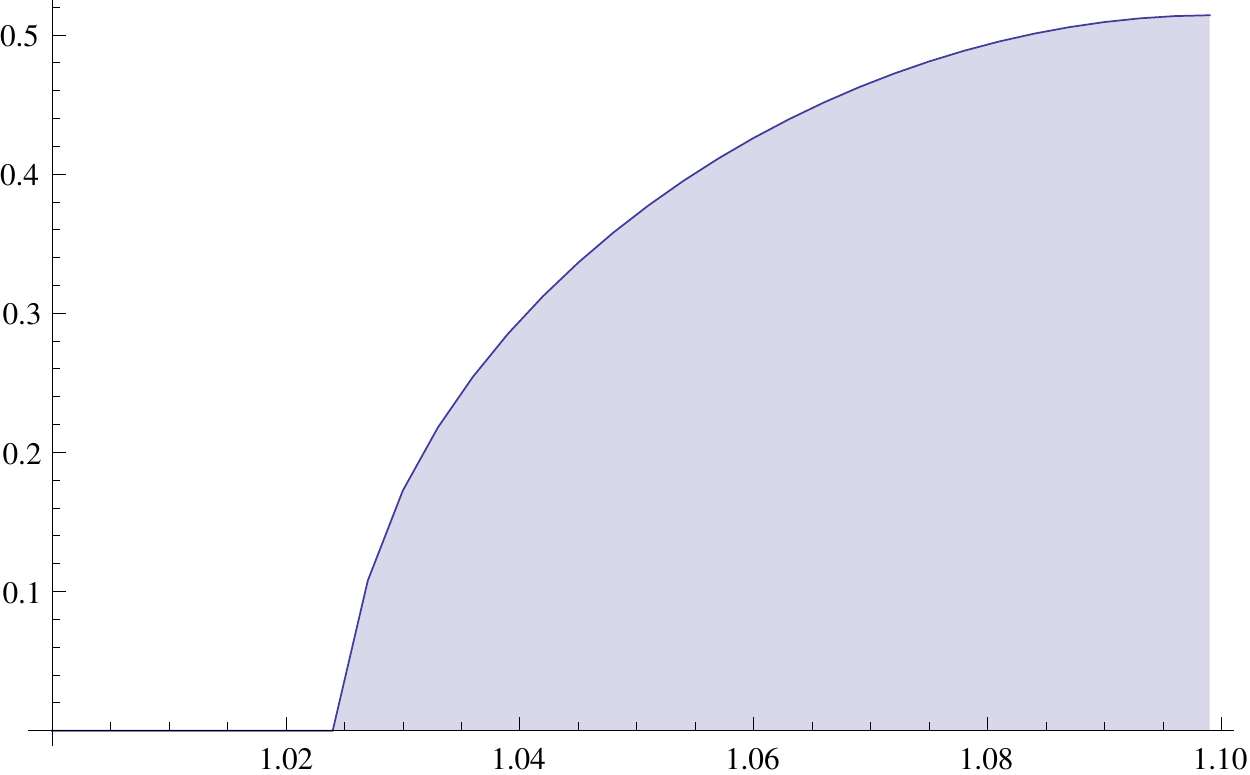} \includegraphics[width=2in]{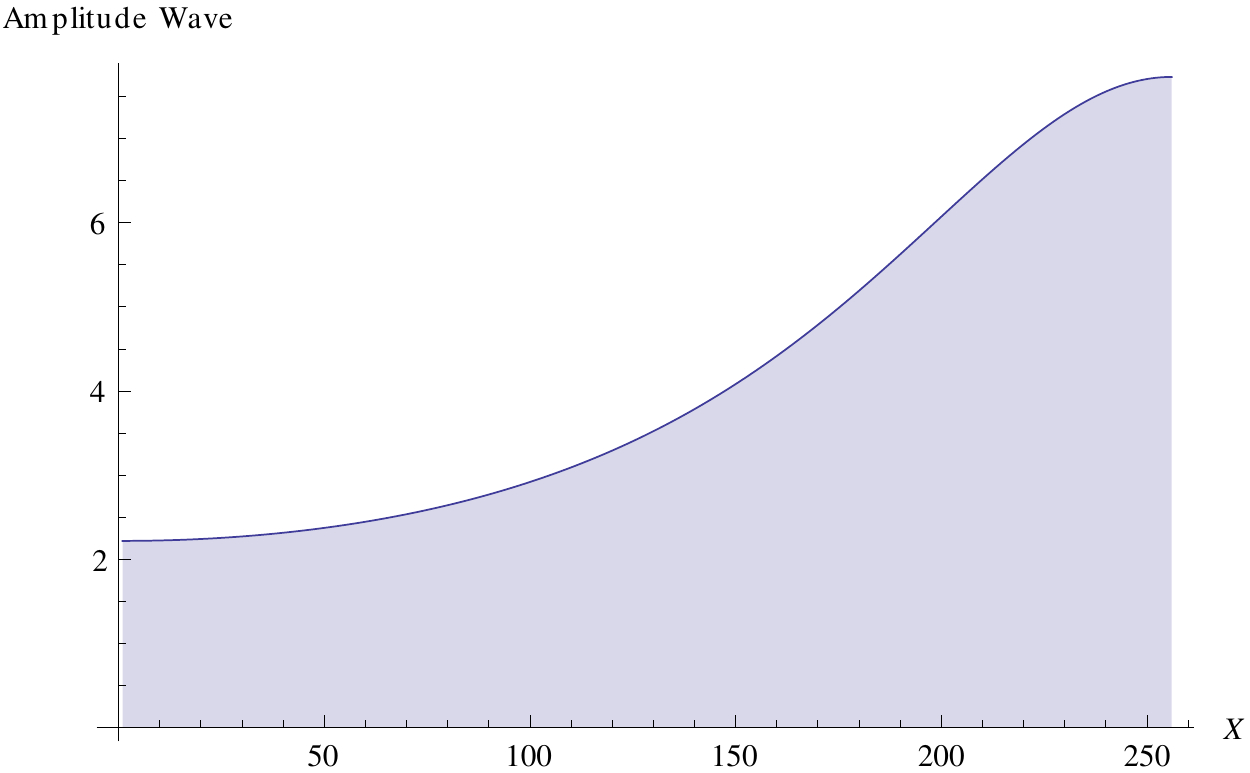} \\
\caption{The graph on the left depicts the imaginary part of the eigenvalue as a function of $\alpha$. The graph on the right represents the wave profile at the onset of instability.}
\end{center}
\end{figure}

\medskip

The Petviashvili method works reasonably well in general but has a number of drawbacks. 
For many parameter values, it either fails to converge or converges to the no-wave solution. 
Furthermore, the convergence of the method is governed by eigenvalues of 
the associated linearized operator and it is unclear if this somehow biases solutions one sees. 
It would be interesting to study stability and its transitions for \eqref{E:KdV} (or \eqref{E:KdV1})
using a more sophisticated numerical method. 
\section{Application: general nonlinearities}\label{S:gKdV}
We shall discuss how the developments in Section~\ref{S:theory} and Section~\ref{S:KdV} 
may be adapted to the KdV equation with fractional dispersion and the general nonlinearity
\begin{equation}\label{E:fKdV}
u_t -\Lambda^\alpha u_x+f(u)_x=0,
\end{equation}
where $0<\alpha\leq 2$ and $f$ is of $C^2$. 

\medskip

Notice that \eqref{E:fKdV} possesses three conserved quantities (abusing notation)
\begin{align}
H(u)=& \int^T_0 \Big( \frac12u \Lambda^\alpha u-F(u)\Big)~dx=:K(u)+U(u), \label{D:fH}
\intertext{where $F'=f$, and}
P(u)=& \int^T_0 \frac12u^2~dx, \label{D:fP}\\
M(u)=& \int^T_0 u~dx. \label{D:fM}
\end{align}
They correspond, respectively, to the Hamiltonian and the momentum, the mass; 
$K$ and $U$ correspond, respectively, to the kinetic and potential energies.
Throughout the section we use $H$, $K$, $U$ and $P$, $M$ 
for those in \eqref{D:fH} and \eqref{D:fP}, \eqref{D:fM}. 

Clearly \eqref{E:fKdV} is in the Hamiltonian form \eqref{E:equation}, for which $J=\partial_x$.
Notice that $H$, $P$, $M$ are smooth in an appropriate subspace of $H^{\alpha}_{per}([0,T])$
and invariant under spatial translations. Clearly they satisfy \eqref{E:PM}.

\medskip

Assume that \eqref{E:fKdV} admits a smooth, four-parameter family of periodic traveling waves,
denoted $u(\cdot+x_0; c,a,T)$, where $c$ and $a$ form an open set in $\mathbb{R}^2$, 
$x_0\in \mathbb{R}$ and $T>0$ are arbitrary, and $u$ is even and $T$-periodic, 
satisfying by quadrature that
\begin{equation}\label{E:ftraveling}
\Lambda^\alpha u- f(u) - c u - a =0
\end{equation}
(in the sense of distributions), or equivalently (abusing notation) 
\begin{equation}\label{E:fE}
\delta E(u; c,a):=\delta(H(u)-cP(u)-aM(u))=0.
\end{equation}
For a broad range of $\alpha$ and nonlinearities, 
the existence of periodic traveling waves of \eqref{E:fKdV} may follow from variational arguments, 
e.g., the mountain pass theorem applied to a suitable variational problem 
whose critical point satisfies \eqref{E:ftraveling}. 
Assume that a periodic traveling wave $u=u(\cdot+x_0; c, a, T)$ of \eqref{E:fKdV} 
satisfies Assumption~\ref{A:kerL0}. 
We shall address its spectral instability near the origin to long wavelengths perturbations. 

\medskip

In particular we follow the approach in Section~\ref{S:theory} and Section~\ref{SS:computation}
and we calculate the modulational instability index $\Delta$, defined in \eqref{D:index}, 
in terms of inner products between
\begin{subequations}
\begin{alignat}{2}
v_1 &= u_a, \qquad &&w_1=M_cu-P_c,\label{E:fvw1} \\
v_2 &= u_x, \qquad &&w_2= \partial_x^{-1} (M_a u_c - M_cu_a),\label{E:fvw2}\\
v_3 &= u_c, \qquad &&w_3=P_a-M_au, \label{E:fvw3}
\end{alignat}
\end{subequations}
together with $L_0^{-1}$, where 
\begin{subequations}
\begin{align}
L_0:=&\partial_x(\Lambda^\alpha-f'(u)-c) \label{E:fL0}
\intertext{and}
L_1:=&[L_0,x]=(\alpha+1)\Lambda^\alpha-f'(u)-c, \label{E:fL1} \\
L_2:=&[L_1,x]=\alpha(\alpha+1)\Lambda^{\alpha-2}\partial_x. \label{E:fL2}
\end{align}\end{subequations}
We shall ultimately express the index in terms of $K$ and $U$, $P$, $M$, 
together with their derivatives with respect to $c$, $a$ as well as $\Omega=1/T$.
Notice that $v_j$'s and $w_j$'s satisfy 
\eqref{E:vw1}-\eqref{E:vw3}, \eqref{E:G} and \eqref{E:parity}. 

\medskip

Below we extend Lemma~\ref{L:identities} and develop integral identities 
that a periodic solution of \eqref{E:ftraveling}, or equivalently \eqref{E:fE}, a priori satisfies. 

\begin{lemma}
If $u$ is $T$-periodic and satisfies \eqref{E:ftraveling} (in the sense of distributions)~then 
\begin{align}
&\int^T_0 f(u)~dx \,+\, c M + a T =0,\label{E:fI1} \\
2K-&\int^T_0 uf(u)~dx-2cP-aM=0, \label{E:fI2}\\
(\alpha-1)K-&\int^T_0 F(u)~dx\,+\,cP+aM+TE_T=0.\label{E:fI3}
\end{align}
\end{lemma}

Multiplying \eqref{E:ftraveling} by $1, u, xu_x+Tu_T$, respectively, 
and integrating over the periodic interval $[0,T]$, the proof is similar to that of Lemma~\ref{L:identities}. 
Hence we omit the detail.

\medskip

In the case of power-law nonlinearities, $uf(u)$ and $F(u)$ are proportional.
Therefore one may relate, e.g., the kinetic energy
to the potential energy, the momentum, the mass; see Lemma~\ref{L:identities}.
In the case of general nonlinearities, for which scaling invariance is lost, on the other hand,
the kinetic and potential energies, the momentum, the mass are no longer linearly dependent. 
Hence the modulational instability index will depend upon all $K, U, P, M$. 
Furthermore their derivatives with respect to $c, a, T$ are not linearly dependent. 
Hence the index will depend upon $K, U, P, M$, 
together with their derivatives with respect to $c, a, T$.

\medskip

We promptly rewrite $\l w_j, L_1v_k\r$, $j,k=1,2,3$, 
in terms of $K, U, P, M$ as functions of $c,a,T$. 
Differentiating \eqref{E:fI2} and \eqref{E:fI1} with respect to $a$ we obtain that 
\[ 
2K_a+U_a-\int^T_0uf'(u)u_a~dx-2cP_a-(aM)_a=0\text{ and }
\int^T_0f'(u)u_a~dx+cM_a+T=0,
\]
respectively. Since $L_1$ is self-adjoint and since
\[ 
L_1u=(\alpha+1)\Lambda^\alpha u-f'(u)-cu\quad\text{and}\quad L_11=-f'(u)-c,
\]
we substitute \eqref{E:fvw1} and calculate that
\begin{subequations} 
\begin{align}
\l w_1,L_1v_1\r =&\l L_1w_1, v_1\r =M_c\l L_1u, u_a\r-P_c\l L_11,u_a\r \notag \\
=&M_c\Big((\alpha+1)K_a-\int^T_0uf'(u)u_a~dx-cP_a\Big)
-P_c\Big(-\int^T_0f'(u)u_a~dx-cM_a\Big) \hspace*{-.2in}\notag \\
=&M_c((\alpha-1)K_a-U_a+cP_a+(aM)_a)-P_cT \notag \\
=&M_c(\alpha K_a-E_a)-P_cT. \label{E:f11}
\end{align}
The last equality utilizes $E=K+U-cP-aM$. Similarly
\begin{align}
\l w_1,L_1v_3\r =&\quad\,M_c(\alpha K_c-E_c+P)-P_cM, \label{E:f13}\\
\l w_3,L_1v_1\r =&-M_a(\alpha K_a-E_a)+P_aT, \label{E:f31}\\
\l w_3,L_1v_3\r =&-M_a(\alpha K_c-E_c+P)+P_cM. \label{E:f33}
\end{align}
\end{subequations}
Moreover we substitute \eqref{E:fvw2} and make an explicit calculation to obtain that
\begin{align}
\langle w_2, L_1 v_2 \rangle=&
\l \partial_x^{-1}(M_au_c-M_cu_a), ((\alpha+1)\Lambda^\alpha-f'(u)-c)u_x\r \notag \\
=&\l \partial_x^{-1}(M_au_c-M_cu_a), \partial_x((\alpha+1)\Lambda^\alpha u-f(u)-cu)\r \notag \\
=&-\l M_au_c-M_cu_a,\, \alpha f(u)+\alpha cu+(\alpha+1)a\r \notag \\
=&-\alpha (\{M,U\}_{c,a}-cG).\label{E:f22}
\end{align}

\medskip

To proceed, differentiating \eqref{E:fI2} and \eqref{E:fI1} with respect to $x\partial_x+T\partial_T$ 
leads, with help of \eqref{E:Kphi} and \eqref{E:Fphi}, to that
\[
2\alpha K-2(\Omega K)_\Omega-(\Omega U)_\Omega 
-\int^T_0 uf'(u)(xu_x+Tu_T)~dx+2c(\Omega P)_\Omega+a(\Omega M)_\Omega=0
\]
and
\[
\int^T_0f'(u)(xu_x+Tu_T)~dx-c(\Omega M)_\Omega=0,
\]
respectively. We then use \eqref{E:pohozaev2} and \eqref{E:Kphi}, \eqref{E:Fphi} to calculate that
\begin{subequations}
\begin{align}
\langle w_1 L_1 L_0^{-1} L_1 v_2 \rangle 
=&\l L_1w_1,xu_x+Tu_T\r-G^{-1}\l w_1,xu_x+Tu_T\r \l L_1w_1,v_1\r \notag \\
=&M_c \l L_1u, xu_x+Tu_T\r-P_c\l L_11,xu_x+Tu_T\r \notag \\
&-G^{-1}(M_c\l u, xu_x+Tu_T\r-P_c\l 1, xu_x+Tu_T\r)\l L_1w_1,v_1\r \notag \\
=&M_c(\alpha(\alpha-1)K-(\Omega(\alpha K-E))_\Omega)\label{E:f1-12} \\
&-G^{-1}(M_c(\Omega P)_\Omega-P_c(\Omega M)_\Omega)(M_c(\alpha K_a-E_a)-P_cT). \notag 
\hspace*{-.3in}
\end{align}
Similarly
\begin{align}
\langle w_3, L_1 L_0^{-1} L_1 v_2 \rangle =
&\l L_1w_3,xu_x+Tu_T\r-G^{-1}\l w_1, xu_x+Tu_T\r \l L_1w_3,v_1\r \notag \\
=&-M_a(\alpha(\alpha-1)K-(\Omega(\alpha K-E))_\Omega)\label{E:f3-12} \\
&+G^{-1}(M_c(\Omega P)_\Omega-P_c(\Omega M)_\Omega)(M_a(\alpha K_a-E_a)-P_aT). \notag 
\hspace*{-.3in}
\end{align}
\end{subequations}
Note from \eqref{E:fL2} that 
\begin{equation}\label{E:L2ux}
\langle w_1, L_2 v_2 \rangle =-2 \alpha (\alpha+1) M_c K\quad\text{and}\quad
\langle w_3, L_2 v_2 \rangle =2 \alpha (\alpha+1) M_a K. 
\end{equation}

\medskip

To summarize, the effective dispersion matrix, defined in \eqref{D:dispersion}, is 
\[
{\bf D} = \left(\begin{matrix}0 & D_{12} & D_{13} \\
1 & D_{22} & 0 \\ D_{31} & D_{32} & D_{33} \end{matrix}\right),
\]
where, substituting 
\eqref{E:f11}-\eqref{E:f33}, \eqref{E:f22} and \eqref{E:f1-12}, \eqref{E:f3-12}, \eqref{E:L2ux} 
into \eqref{D:b22}, \eqref{D:c22}, \eqref{D:b23} and \eqref{D:b32}, \eqref{D:c32}, \eqref{D:b33},
we make an explicit calculation to obtain that 
\begin{align*}
D_{12}=&-2\alpha^2GM_aK-G(\Omega(\alpha K-E))_\Omega \\
&+M_c((\Omega P)_\Omega-P_c(\Omega M)_\Omega)(M_a(\alpha K_a-E_a)-P_aT), \\
D_{13} =&-M_a(\alpha K_a-E_a)+ P_a T, \\
D_{22} =&-\alpha(\{M,U\}_{c,a}-cG)-M_a(\alpha K_c-E_c+P)+P_aM, \\
D_{31} =&  M_c(\alpha K_c-E_c+P)-P_cM, \\
D_{32} =&2\alpha^2GM_cK-G(\Omega(\alpha K-E))_\Omega \\
&-M_c(\Omega P)_\Omega-P_c(\Omega M)_\Omega)(M_c(\alpha K_a-E_a)-P_cT)\\
&+(M_c(\alpha K_c-E_c+P)-P_cM) \\
&\qquad\cdot (\alpha\{M,U\}_{c,a}-cG+M_a(\alpha K_c-E_c+P)+P_cM),\\ 
D_{33} =&M_c(\alpha K_a-E_a)-P_cT.
\end{align*}
Furthermore a complex eigenvalue of ${\bf D}$ implies modulational instability. 

\begin{appendix}

\section{Proof of Lemma \ref{lem:index}}\label{A:proof}
We first prove \eqref{E:index} in finite dimensions. 
Suppose that $\vec e_1, \vec e_2, \ldots, \vec e_n$ form a basis of ${\mathbb R}^n$
and $\vec e_1, \vec e_2, \ldots, \vec e_k$ form a basis of $S \subset \mathbb{R}^n$. 
Recall from the Jacobi-Sturm sequence argument that 
the number of negative eigenvalues of an $n\times n$ matrix ${\bf M}=(m_{ij})_{i,j=1}^n$ 
is equal to the number of sign changes in 
\[
1,m_1,m_2,\ldots, m_n,
\]
where $m_k$ is $k$-th principal minor, defined as 
\[
m_k = \det(m_{ij})_{i,j=1}^k.
\]
The number of negative eigenvalues of ${\bf M}|_S$ is, similarly, 
equal to the number of sign changes in 
\[
 1,m_1,m_2,\ldots, m_k.
\]
Moreover recall from a duality formula for minor determinants that 
the number of negative eigenvalues of $({\bf M}^{-1})|_{S^\perp}$ is 
 \[
 1,\frac{m_{n-1}}{m_n},\frac{m_{n-2}}{m_n},\ldots, \frac{m_{k}}{m_n}.
\]
Therefore \eqref{E:index} follows.
(Let $I$ and $J$ be subsets of $\{1,2,\ldots,n\}$ of size $k$, i.e., $|I|=|J|=k$ and let
$\det_{I,J}({\bf M}) = \det(m_{i,j})_{i\in I,j\in J}$.
If $I',J'$ denote the complementary sets to $I,J$ then  
\[
\det_{I,J}({\bf M}^{-1}) = (-1)^{(\sum_{i \in I} i + \sum_{j \in J} j)} \frac{\det_{J',I'}({\bf M})}{\det({\bf M})}.
\]
To interpret, inverse matrices take $k$-minor determinants to complementary $(n-k)$-minor determinants, like the Hodge-* operator acts on $k$-forms. 
A proof based upon the Schur complement formula may be found in \cite[pp.~41]{DSerre}.)

\medskip

To proceed, let ${\bf M}$ be invertible and bounded below with compact resolvent. 
Let $\mu_1 \leq \mu_2 \leq \mu_3 \ldots $ denote 
eigenvalues of ${\bf M}$, and $v_1,v_2,v_3,\ldots $ be the corresponding eigenvectors.
Let $S_n$ denote the subspace spanned by $v_1,v_2,\ldots, v_n$
and let ${\bf M}_n=\Pi_{S_n} {\bf M} \Pi_{S_n}$ denote the symmetric projection of $M$ onto $S_n$. 
Then,
\begin{itemize}
\item[(M1)] ${\bf M}_n \to {\bf M}$ as $n\to \infty$ in the strong operator topology 
and $n_-({\bf M}_n) = n_-({\bf M})$ for $n$ sufficiently large;
\item[(M2)] if ${\bf M}$ is invertible then ${\bf M}_n:S_n \to S_n$ is invertible for $n$ sufficiently large;
\item[(M3)] $({\bf M}_n)^{-1}\to {\bf M}^{-1}$ in the uniform operator topology and  
 $n_-(({\bf M}_n)^{-1}|_{S\perp})= n_-({\bf M}^{-1}|_{S^\perp})$ for $n$ sufficiently large;
\item[(M4)] ${\bf M}|_S= {\bf M}_n|_S$. 
\end{itemize}
Therefore \eqref{E:index} follows from a limiting argument.

\end{appendix}

\subsection*{Acknowledgements}
JCB is supported by the National Science Foundation under grant No. DMS--1211364 and 
by a Simons Foundation fellowship. He would like to thank the Department of Mathematics at MIT
for its hospitality during the writing of the paper. 
VMH is supported by the National Science Foundation under grant No. DMS--1008885, 
the University of Illinois at Urbana-Champaign under the Campus Research Board grant No. 11162 
and by an Alfred P. Sloan research fellowship. 

\bibliographystyle{amsalpha}
\bibliography{modulationalBib}
\end{document}